\documentclass[11pt]{article}
\usepackage{graphicx, float} 
\usepackage{amsmath,amssymb,amsthm}
\usepackage{fullpage}
\usepackage[hidelinks]{hyperref}
\usepackage[nameinlink]{cleveref}
\usepackage{pgfplots}
\pgfplotsset{compat=1.18}
\usepackage{mathrsfs}
\usetikzlibrary{arrows}
\usetikzlibrary{patterns}

\usepackage{todonotes,wrapfig,textcomp,array}
\usepackage{color,xcolor}

\usepackage{lineno}

\newtheorem{theorem}{Theorem}
\newtheorem{claim}[theorem]{Claim}
\newtheorem{definition}[theorem]{Definition}
\newtheorem{observation}[theorem]{Observation}
\newtheorem{proposition}[theorem]{Proposition}

\newcommand\conv{\ensuremath{\mathrm{conv}}}

\title{Erd\H{o}s-Szekeres Maker-Breaker Games\thanks{A preliminary version of this paper appeared in the \emph{Proceedings of 31st International Computing and Combinatorics Conference (Chengdu, 2025), LNCS 15983, Springer, pp.\ 248--261.}}}
\author{Aleksa D\v{z}uklevski\footnote{Department of Applied Mathematics, Faculty of Mathematics and Physics, Charles University, Czech Republic, and Department of Mathematics and Informatics, Faculty of Sciences, University of Novi Sad, Serbia Email: \texttt{aleksa@matfyz.cuni.cz}. Supported by GA\v{C}R grant 23-04949X and GAUK grant VV–2025–260822. } \and Dömötör P\'{a}lvölgyi\footnote{ELTE Eötvös Loránd University and HUN-REN Alfréd Rényi Institute of Mathematics, Budapest, Hungary. Email: \texttt{domotor.palvolgyi@ttk.elte.hu}. Supported by the NRDI EXCELLENCE-24 grant no.~151504 Combinatorics and Geometry and by the ERC Advanced Grant no.~101054936 ERMiD.} \and Alexey Pokrovskiy\footnote{Department of Mathematics, University College London, UK. Email: \texttt{dralexeypokrovskiy@gmail.com.}} \and Csaba D. T\'{o}th\footnote{Department of Mathematics, California State University Northridge, Los Angeles, CA, and Department of Computer Science, Tufts University, Medford, MA, USA. Email: \texttt{csaba.toth@csun.edu}. Research supported in part by the NSF award DMS-2154347.} \and Tom\'{a}\v{s} Valla\footnote{Faculty of Information Technology, Czech Technical University in Prague, Czech Republic. Email: \texttt{tomas.valla@fit.cvut.cz}. Research supported by the Czech Science Foundation Grant no. 24-12046S.} \and Lander Verlinde\footnote{School of Computer Science, University of Auckland, New Zealand. Email: \texttt{lander.verlinde@pg.canterbury.ac.nz}. Research supported by the New Zealand Marsden Fund.}}
\date{}

\begin{document}

\definecolor{ududff}{rgb}{0.30196078431372547,0.30196078431372547,1}
\definecolor{qqttcc}{rgb}{0,0.2,0.8}
\definecolor{ccqqqq}{rgb}{0.8,0,0}
\definecolor{qqzzqq}{rgb}{0,0.6,0}
\definecolor{ttzzqq}{rgb}{0.2,0.6,0}
\definecolor{ffwwqq}{rgb}{1,0.4,0}
\definecolor{qqqqff}{rgb}{0,0,1}
\definecolor{qqwuqq}{rgb}{0,0.39215686274509803,0}
\definecolor{uuuuuu}{rgb}{0.26666666666666666,0.26666666666666666,0.26666666666666666}
\definecolor{ttffqq}{HTML}{0a69aa}
\definecolor{ffqqqq}{HTML}{d6bf0f}
\definecolor{zzttqq}{rgb}{0.96,0.26,0.91}
\definecolor{ffzzcc}{rgb}{1,0.6,0.8}
\definecolor{ccffww}{rgb}{0.8,1,0.4}
\definecolor{zzccff}{rgb}{0.6,0.8,1}
\definecolor{yqqqyq}{rgb}{0.5,0,0.5}
\definecolor{ccffww}{rgb}{0.8,1,0.4}
\definecolor{zzffff}{rgb}{0.6,1,1}
\definecolor{xdxdff}{rgb}{0.5,0.5,1}
\definecolor{qqffqq}{rgb}{0,1,0}
\definecolor{ffqqqq}{rgb}{1,0,0}

\maketitle
\begin{abstract}
    We present
    new results on
    Maker-Breaker games arising from the Erd\H{o}s-Szekeres problem in planar geometry. This classical problem asks how large a set in general position has to be to ensure the existence of $n$ points that are the vertices of a convex $n$-gon. Moreover, Erd\H{o}s further extended this problem by asking what happens if we also require that this $n$-gon has an empty interior.
    In a 2-player Maker-Breaker setting, this problem inspires two main games. In both games, Maker tries to obtain an empty convex $k$-gon, while Breaker tries to prevent her from doing so. The games differ only in which points can comprise the winning $k$-gons: in the monochromatic version the points of both players can make up a $k$-gon, while in the bichromatic version only Maker's points contribute to such a polygon. Both settings are studied in this paper. We show that in the monochromatic game, Maker always wins. Even in a biased game where Breaker is allowed to place $s$ points per round, for any constant $s \geq 1$, Maker has a winning strategy. In the bichromatic setting, Maker still wins whenever Breaker is allowed to place $s$ points per round for any constant $s<2$. This settles an open problem posed by Aichholzer et al.~(2019).
    Furthermore, we show that there are games that are not a lost cause for Breaker. Whenever $k\ge 8$ and Breaker is allowed to play 12 or more points per round, she  has a winning strategy. We also consider the one-round bichromatic game (a.k.a.\ the offline version). In this setting, we show that Breaker wins if she can place twice as many points as Maker but if the bias is less than $2$, then Maker wins for large enough set of points.\\

    \noindent\textbf{Keywords.} Erd\H{o}s-Szekeres theorem
maker-breaker game, convex $k$-hole.
\end{abstract}

\section{Introduction}

The Erd\H{o}s-Szekeres theorem~\cite{ESZ35} is a classical result in Ramsey theory. It states that every set of $n$ points in the plane in general position contains a subset of $\Omega(\log n)$ points in convex position. A set of points in the plane is in \emph{general position} if no three points are collinear, and they are in \emph{convex position} if they are the vertices of a convex polygon. The bound $\Omega(\log n)$ is the best possible; it has been refined to $(1-o(1))\log_2 n$ and significant efforts have been devoted to finding precise bounds~\cite{HMPT20,Suk17}. The Erd\H{o}s-Szekeres theorem had lasting impact on combinatorial geometry---many generalizations and variations have been studied.

For a point set $S\subset \mathbb{R}^2$ and an integer $k\geq 3$, a \emph{$k$-hole} is a subset $H\subset S$ of size $k$ such that $\conv(H)$ is a convex $k$-gon and $\conv(H)\cap S=H$, where $\conv(H)$ denotes the convex hull of $H$. Let $h(k)$ denote the minimum integer $n$ such that every set of $n$ points in the plane in general position contains a $k$-hole. Horton~\cite{Horton1983} constructed point sets in general position that do not contain any 7-hole, which implies $h(k)=\infty$ for all $k\geq 7$. It is easy to see that $h(3)=3$, $h(4)=5$, and Harboth~\cite{GDZPPN002079801} showed that $h(5)=10$. Gerken~\cite{Gerken08} and Nicol\'as~\cite{Nicolas07} proved independently that $h(6)<\infty$. With a computer assisted proof, Heule and Scheucher~\cite{HS24} recently showed that $h(6)=30$; completely answering a question posed by Erd\H{o}s~\cite{Erd78}.
The minimum and maximum number of $k$-holes among $n$ points in the plane were also studied~\cite{AichholzerBHKPS20,AichholzerMGHHH15,BalkoSV23,BV04,PinchasiRS06},
as well as the associated counting and enumeration problems~\cite{Bae22,DobkinEO90,MitchellRSW95}.

Competitive games between two players
to achieve a geometric structure were studied previously~\cite{abdhkkrsu-gt-05,hefetz2014positional}.
The natural competitive game arising from the Erd\H{o}s-Szekeres problem is the
endeavor of two players to achieve a $k$-hole for a given positive integer $k$
by alternately placing points in general position in the plane.
Depending on the goal of the game, three variants of the two-player game spawn
from the Erd\H{o}s-Szekeres problem.

\begin{enumerate}
    \item Both players want to obtain a $k$-hole (\emph{Maker-Maker}).
          Whoever creates the first $k$-hole wins.
    \item Both players want to avoid a $k$-hole (\emph{Avoider-Avoider}).
          Whoever creates the first $k$-hole loses.
    \item The first player wants to obtain a $k$-hole and the second wants to prevent this (\emph{Maker-Breaker}). The first player (Maker) wins if a $k$-hole is ever created by either player, and the second player (Breaker) wins if she is able to prolong the game indefinitely.
\end{enumerate}

The \emph{Erd\H{o}s-Szekeres Maker-Maker game} was introduced by Valla~\cite{Valla06}. Kolipaka and Govindarajan~\cite{KolipakaG13} studied the Avoider-Avoider game
and showed that for $k=5$ the second player can win within 9 rounds.
Aichholzer et al.~\cite{ADH+19}
proposed the Maker-Breaker game for $k\geq 7$. For $k\leq 6$, Maker inevitably wins as every sufficiently large point set in general position contains a $k$-hole. For $k=7$, Maker can easily add one more point to extend a 6-hole to a 7-hole. Das and Valla~\cite{DV26} recently studied the Maker-Breaker game: they showed that Maker (i.e., the first player) has a winning strategy for all $k\leq 8$. If $r(k)$ denotes the minimum number of rounds for a winning strategy for Maker, then $r(3)=2$ and $r(4)=3$ are obvious as Maker places the $(2r-1)$st point in round $r$; and they proved that $r(5)\leq 6$, $r(7)\leq 8$, and $r(8)\leq 12$. They asked what the maximum $k$ is for which Maker has a winning strategy.

Aichholzer et al.~\cite{ADH+19}
also introduced several \emph{bichromatic} versions of Erd\H{o}s-Szekeres games. In the bichromatic Maker-Breaker game, the set $S\subset \mathbb{R}^2$ is partitioned as $S=M\cup B$, where $M$ and $B$ denote the points played by Maker and Breaker, respectively. Maker wins if $M\cup B$ contains a $k$-hole $H$ \emph{and} $H\subset M$ (that is, the hole consists of Maker's points, and $\conv(H)$ contains neither Breaker points nor any additional Maker points).

The one-round game (a.k.a.\ the \emph{offline version}) has also been considered.
Conlon and Lim~\cite{ConlonL23} proved that for every set $M$ of $m$ points in the plane in general position, there exists a set $S=M\cup B$ without 9-holes; however the set $B$ is much larger than $M$ in their construction. In the bichromatic variant, Maker places a set $M$ of $m$ points in the plane in general position, and then Breaker places a set $B$ of points to prevent a $k$-hole $H\subset M$ with respect to $S=M\cup B$. For $k\in \{3,4\}$, $|B|=2m-O(1)$ points are always sufficient and sometimes necessary~\cite{CzyzowiczKU00,SU07}; the current best lower bound is $|B|\geq \frac29m-O(1)$ for $k=5$~\cite{CanoOHSTU15}.

By default, each player places one new point per round in the online version. However, to balance the fortunes of the players, one might introduce \emph{bias} by allowing one of the players to place more points than the other. The number of points that a player is allowed to place shall also be referred to as \emph{speed}. We shall denote the respective speeds of Maker and Breaker by $s_M:s_B$, normalized to $s_M=1$ or $s_B=1$. In particular, the default version, where the players have equal speed, is the $1:1$ game.
If Maker or Breaker have noninteger speeds $s_M:s_B$, it is conventional to order the elements of the multiset $\{i \cdot s_M:i\in \mathbb N\}\cup\{i \cdot s_B:i\in \mathbb N\}$ in increasing order (breaking ties in favor of Maker), and use this order to determine whose turn it is (see e.g.~\cite[Section 4]{kutz2005angel}).

\paragraph{Our results.}
We present three new strategies that can be used for different Maker-Breaker games. The first one is based on maintaining a structure which we shall call a \emph{$k$-strip}: Intuitively, it can be seen as a set of $k$ points on an $x$-monotone concave arc such that the region below the arc is empty  (\Cref{sec:monochromatic}).
We show that in the monochromatic $1:1$ game, Maker is able to construct such a $k$-strip of arbitrary size and describe an algorithm to do so. Since such a $k$-strip is a specific type of $k$-hole, this proves that Maker always wins the monochromatic game (\Cref{thm:strips}), settling the question of Das and Valla \cite{DV24}.
The main advantage of these strips is that they provide the maker with a lot of `space' to complete a $k$-hole. The strips can be constructed parallel to each other, so that neighboring strips require multiple points to be blocked by Breaker, while Maker can use any one of them to win the game. Using these $k$-strips does mean that Maker needs to place a lot of points to construct all the neighboring strips, which raises follow-up questions regarding efficiency of the strategy.
Furthermore, we show that even if bias is introduced in favor of Breaker, Maker can still construct $k$-strips. Hence, Maker wins the monochromatic game even with speed $1:s_B$ for any constant $s_B \geq 1$ (\Cref{thm:MB}). A last adaption of the algorithm also shows that Maker has a winning strategy for the bichromatic game with bias $1:s_B$ for any constant $s_B < 2$ (\Cref{thm:s2}). Hence, this strategy is versatile, and with the right adjustments, Maker is able to use it to win multiple types of Maker-Breaker games.

Another strategy is based on \emph{perturbed polygons} (\Cref{subsec:Breaker}). This is a strategy designed for Breaker in the bichromatic game with large enough bias in her favor. In this strategy, Breaker places her points as vertices of a slightly perturbed regular polygon inscribed in sufficiently small circles around Maker's points. Intuitively, one can think of this strategy of using the bias to `fence in' Maker's points and hence not allow a $k$-hole to be constructed. The strategy relies on surrounding every new point Maker plays, and possibly another nearby point. Ideally, these surrounding points are placed in a regular polygon, making the analysis easier. However, this sometimes would lead to three points being collinear (contrary to the general position requirement), and thus the points need to be perturbed. Hence, this strategy relies on a suitable choice of radii to box in Maker's points, and a specific type of perturbation. Both elements of the strategy are explained in \Cref{subsec:Breaker}. We prove that whenever $k \geq 8$, Breaker is able to prevent Maker from constructing a $k$-hole in the bichromatic game with speed $1:s_B$ for $s_B \geq 12$ (\Cref{thm:bichromatic}). We also provide a trade-off between $k$ and the bias: Breaker wins the bichromatic game with $1: 2 \lambda$ bias for every integer $\lambda \geq 3$ if $k \geq 3 \lceil 2\lambda/(\lambda -2)\rceil -1$ (\Cref{thm:bichromatic+}).

Lastly, we propose a strategy for Maker for the one-round bichromatic game whenever the bias in favor of Breaker is strictly less than $2$ (\Cref{ssec:HalesJewett}). Maker's strategy builds on the Density Hales-Jewett Theorem \cite{FK91} from combinatorics. In particular, Maker starts with a section of the integer lattice $\mathbb{Z}^d$, for a sufficiently large dimension $d$, mapped with a suitable chosen generic linear transformation onto a plane and perturbed to be in general position. Importantly, the perturbation ensures that some of the collinear $t$-tuples of points of the grid are mapped into a $t$-hole (with respect to Maker's points). We prove that Maker wins the bichromatic one-round game with bias $1:(2-\varepsilon)$ for every $\varepsilon > 0$ (\Cref{thm:grid}). This is the best possible, as Breaker wins the bichromatic one-round game with bias $1:s_B$ for $s_B \geq 2$ (\Cref{obs:bichrom}).

\section{Monochromatic Maker-Breaker Games}
\label{sec:monochromatic}

When proving their result, Erd\H{o}s and Szekeres~\cite{ESZ35} have introduced a notion that later became known as \textit{caps and cups}.

\begin{definition}
    A set  $\{p_1, \dots, p_k\}$ of $k$ points, labeled in increasing order by their $x$-coordinates,
    forms a \emph{$k$-cap} (\emph{$k$-cup}) if the points are in convex position
    and their convex hull is bounded from below (above) by a single edge; see \Cref{fig: cup and cap}.
\end{definition}

\begin{figure}[htp]
\centering
\begin{tikzpicture}[line cap=round,line join=round,>=triangle 45,x=1cm,y=1cm,scale=0.8]
\clip(0,-1) rectangle (13.2,3);
\fill[line width=1pt,color=qqqqff,fill=green,fill opacity=0.1]
    (9,-1.5) -- (9,0.5) -- (10,1.5) -- (11,2) -- (12,1.8) -- (13,0.8) -- (13,-1.2) -- cycle;
\draw [line width=1pt] (1,1)-- (2,2);
\draw [line width=1pt] (2,2)-- (3,1.75);
\draw [line width=1pt] (3,1.75)-- (4,0);
\draw [line width=1pt] (4.5,1)-- (5.5,0);
\draw [line width=1pt] (5.5,0)-- (6.5,-0.56);
\draw [line width=1pt] (6.5,-0.56)-- (7.5,-0.4);
\draw [line width=1pt] (7.5,-0.4)-- (8.5,1.3);
\draw [line width=1pt] (8.5,1.3)-- (4.5,1);
\draw [line width=1pt] (4,0)-- (1,1);
\draw [line width=1pt] (9,0.5) -- (10,1.5);
\draw [line width=1pt] (10,1.5) -- (11,2);
\draw [line width=1pt] (11,2) -- (12,1.8);
\draw [line width=1pt] (12,1.8) -- (13,0.8);
\draw [line width=1pt] (9,0.5) -- (9,-1.5);
\draw [line width=1pt] (13,-1.5) -- (13,0.8);
\draw (11,0.4) node {\large $\conv(P) + \vec{d}$};
\begin{scriptsize}
\draw [fill=ududff] (1,1) circle (2.5pt);
\draw [fill=ududff] (2,2) circle (2.5pt);
\draw [fill=ududff] (3,1.75) circle (2.5pt);
\draw [fill=ududff] (4,0) circle (2.5pt);
\draw [fill=ududff] (4.5,1) circle (2.5pt);
\draw [fill=ududff] (5.5,0) circle (2.5pt);
\draw [fill=ududff] (6.5,-0.56) circle (2.5pt);
\draw [fill=ududff] (7.5,-0.4) circle (2.5pt);
\draw [fill=ududff] (8.5,1.3) circle (2.5pt);
\draw [fill=ududff] (9,0.5) circle (2.5pt);
\draw [fill=ududff] (10,1.5) circle (2.5pt);
\draw [fill=ududff] (11,2) circle (2.5pt);
\draw [fill=ududff] (12,1.8) circle (2.5pt);
\draw [fill=ududff] (13,0.8) circle (2.5pt);
\end{scriptsize}
\end{tikzpicture}
\caption{Example of a 4-cap, a 5-cup and a 5-strip with respect to $\vec{d}=\overrightarrow{(0,-1)}$.}
\label{fig: cup and cap}
\end{figure}
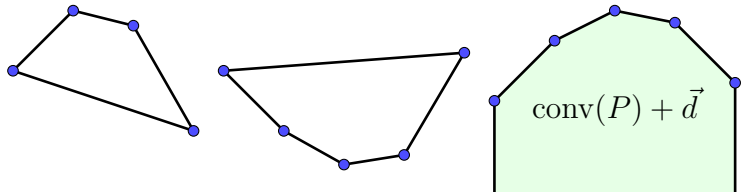

We extend these caps to a new notion of $k$-strips. Intuitively, a $k$-strip is a $(k+1)$-hole with $k$ points in $\mathbb{R}^2$ and one point ``at infinity''. Recall that the Minkowski sum of two sets $A,B\in \mathbb{R}^2$ is defined as $A+B=\{a+b: a\in A, b\in B\}$.  For a nonzero vector $v$, let $\vec{v}$ denote the halfline from the origin in direction $v$. As a shorthand, let $\vec{d}=\overrightarrow{(0,-1)}$ denote a downward vertical~halfline.

\begin{definition}
For every $k \in\mathbb{N}$, a \emph{$k$-strip} with respect to a point set $S\subset \mathbb{R}^2$ is a set $P \subset S$ of $k$ points in convex position such that there exists a halfline $\vec{v}$ for which the Minkowski sum $C=\conv(P)+\vec{v}$ satisfies ${\rm int}(C)\cap S=\emptyset$ and $\partial C\cap S=P$; see an example in \Cref{fig: cup and cap}.  We say that two $k$-strips, $P_1$ and $P_2$, are \emph{parallel} if they are $k$-strips w.r.t.\ the same halfline $\vec{v}$, and $C_1=\conv(P_1)+\vec{v}$ and $C_2=\conv(P_2)+\vec{v}$ are disjoint.
\end{definition}

In particular, if $P$ is a $k$-strip w.r.t.\ the downward halfline $\vec{d}$ (resp., upward halfline $\overrightarrow{(0,1)}$), then $P$ is a $k$-cap (resp., $k$-cup).

We also introduce cones adjacent to $k$-strips for $k\geq 2$. Let $P=\{s_1,s_2,\ldots , s_k\}$ be a $k$-strip w.r.t.\ the downward halfline $\vec{d}$, and assume that the points $s_{1}, \ldots , s_{k}$ are sorted by increasing $x$-coordinates. We define two cones, $C^-$ and $C^+$, as follows (see \Cref{fig: cone definition} for an example). Let $L(s_{1})=s_{1}+\vec{d}$ and $L(s_{k})=s_{k}+\vec{d}$ be the downward halflines starting from $s_{1}$ and $s_{k}$, respectively. Let $L^-$ be the line passing through $s_{1}$ and $s_{2}$, and $L^+$ the line passing through $s_{k-1}$ and $s_{k}$. Now let $C^-$ be the cone swept by rotating the halfline $L(s_{1})$ clockwise about $s_{1}$ until it reaches $L^-$ or it passes through a point in $S$.
Similarly, let $C^-$ be the cone swept by rotating the halfline $L(s_{k})$ counterclockwise about $s_{k}$ until it reaches $L^+$ or it passes through a point in $S$. Since the aperture of each cone is less than $\pi$, both cones are convex.
Observe that this gives us a way of extending a $k$-strip into a $(k+1)$-strip. Indeed, for any point $p\in S\cap({\rm cl}(C^-)\cup {\rm cl}(C^+))$, the set $P\cup \{p\}$ is a $(k+1)$-strip with respect to the downward halfline $\vec{d}$.

    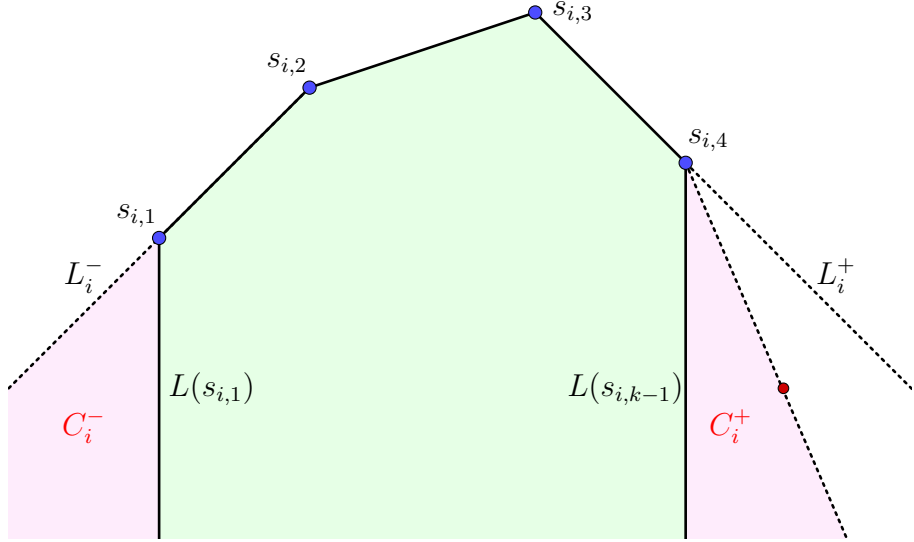
\begin{figure}
        \centering
        \begin{tikzpicture}[line cap=round,line join=round,>=triangle 45,x=1cm,y=1cm]
\clip(-6,-2) rectangle (6,5.5);
\fill[line width=1pt,color=qqqqff,fill=green,fill opacity=0.1] (-4,2) -- (-2,4) -- (1,5) -- (3,3) -- (3,-4) -- (-4,-4) -- cycle;
\fill[line width=1pt,color=zzttqq,fill=zzttqq,fill opacity=0.1] (-4.0,2.0) -- (-10.0,-4.0) -- (-4.0,-4.0) -- cycle;
\fill[line width=1pt,color=zzttqq,fill=zzttqq,fill opacity=0.1] (3.0,3.0) -- (3.0,-4.0) -- (6.0,-4.0) -- cycle;
\draw [line width=1pt] (-4.0,2.0)-- (-2.0,4.0);
\draw [line width=1pt] (-2.0,4.0)-- (1.0,5.0);
\draw [line width=1pt] (1.0,5.0)-- (3.0,3.0);
\draw [line width=1pt] (-4.0,2.0) -- (-4.0,-11.6);
\draw [line width=1pt] (3.0,3.0) -- (3.0,-11.6);
\draw [line width=1pt,dotted,domain=-6:-2.0] plot(\x,{(-12.0-2.0*\x)/-2.0});
\draw [line width=1pt, dotted] (3,3) -- (10,-4);
\draw [line width=1pt, dotted] (3,3) -- (6,-4);
\begin{scriptsize}
\draw [fill=ududff] (-4.0,2.0) circle (2.5pt);
\draw [fill=ududff] (-2.0,4.0) circle (2.5pt);
\draw [fill=ududff] (1.0,5.0) circle (2.5pt);
\draw [fill=ududff] (3.0,3.0) circle (2.5pt);
\draw[color=black]  (-4.3,2.3)  node {\large $s_{i,1}$};
\draw[color=black]  (-2.3,4.3)  node {\large $s_{i,2}$};
\draw[color=black]  (1.5,5)  node {\large $s_{i,3}$};
\draw[color=black]  (3.3,3.3)  node {\large $s_{i,4}$};
\draw[color=black] (-3.3,0) node {\large $L(s_{i,1})$};
\draw[color=black] (2.2,0) node {\large $L(s_{i,k-1})$};
\draw[color=black] (-5,1.5) node {\large $L_i^-$};
\draw[color=black] (5,1.5) node {\large $L_i^+$};
\draw[color= red] (-5,-0.5) node {\large $C_i^-$};
\draw[color= red] (3.6,-0.5) node {\large $C_i^+$};
\draw [fill=ccqqqq] (4.3,0.0) circle (2pt);
\end{scriptsize}
\end{tikzpicture}
        \caption{A 4-cap $S_i=\{s_{i,1},s_{i,2},s_{i,3},s_{i,4}\}$ and 4-strip $\conv(S_i)+\vec{d}$. Cones $C_i^-$ and $C_i^+$ are colored in pink. The half-lines $L(s_{i,1})$ and $L(s_{i,k-1})$, resp., are rotated clockwise and counterclockwise until they pass through a point in the point set or until they hit $L_i^-$ or $L_i^+$.}
        \label{fig: cone definition}
    \end{figure}

\subsection{Monochromatic Maker-Breaker Game without Bias}

In the $1:1$ game, we show by induction that Maker can create arbitrarily many $k$-strips for any $k\in \mathbb{N}$. As a $k$-strip is a $k$-hole, this immediately implies that Maker wins for any $k\in \mathbb{N}$.

\begin{theorem}
\label{thm:strips}
In the monochromatic Maker-Breaker game (without bias),
for every $k, t \in \mathbb{N}$, Maker can ensure that there are $t$ parallel $k$-strips among the points placed by both players. This can be accomplished in $t/2$ rounds for $k=1$ and at most $(\frac53\cdot 4^{k-2}-\frac23)t$ rounds for $k\geq 2$.
\end{theorem}

\begin{proof}
       We proceed by induction on $k$. In the base case, $k=1$, any $t$ points in general position form $t$ parallel $1$-strips. Similarly, for $k=2$, it is easy to see that any  $2t$ points in general position form $t$ parallel $2$-strips.
    We may assume (by rotating the coordinate system, if necessary) that the points have distinct $x$-coordinates. Then we can partition the $2t$ points into $t$ pairs of points with consecutive $x$-coordinates; and use downward $2$-strips.

    For the induction step, we are given $k,t\in \mathbb{N}$, $k\geq 3$, and assume that the claim holds for all $k',t'\in \mathbb{N}$ with $k'<k$. By the induction hypothesis, Maker can form $4t$ parallel $(k-1)$-strips with respect to set $S$ of points played by both players so far. We may assume (by rotating the coordinate system if necessary) that all $4t$ strips
    are downward strips with halfline $\vec{d}=\overrightarrow{(0,-1)}$; and the points in $S$ have distinct $x$-coordinates. In particular, the convex hulls of the $4t$ strips have pairwise disjoint projections to the $x$-axis. Label the strips by $S_1, S_2, \dots, S_{4t}\subset S$ in increasing order by $x$-projections of their convex hulls. Furthermore, for every $i\in [4t]$, label the points in $S_i$ by $s_{i,1},\ldots , s_{i,k-1}$ in increasing order by their $x$-coordinates.

    For every $j\in [2t]$, let $\ell_j$ be a vertical line between $S_{2j-1}$ and $S_{2j}$ that does not pass through any other point in $S$; and let
    \[
    \vec{\ell}_j= \ell_j\cap C^+_{2j-1}\cap C^-_{2j}.
    \]
    Note that $\vec{\ell}_j$ is a downward pointing halfline. Furthermore,
    for any point $p_j\in \vec{\ell}_j$, both $S_{2j-1}\cup \{p_j\}$ and $S_{2j}\cup \{p_j\}$ is a $k$-strip with respect to the point set $S\cup \{p_j\}$.

     In the next $2t$ rounds, Maker places arbitrary points $p_j\in \vec{\ell}_j$ for all $j\in [2t]$ such that they are in general position.
    For every $j\in [2t]$, point $p_j$ creates two \emph{candidates} for possible $k$-strips:
    $S_{2j-1}\cup \{p_j\}$ and $S_{2j}\cup \{p_j\}$. Let $R_{2j-1}= \conv(S_{2j-1}\cup \{p_j\})+\vec{d}$ and $R_{2j}=\conv(S_{2j}\cup \{p_j\})+\vec{d}$; see \Cref{fig:k-strips}.

    Meanwhile, Breaker also plays $2t$ points.
    If Breaker places any point in $\vec{\ell}_j$ below $p_j$, then we replace $p_j$ with Breaker's point in the definition of the candidate regions $R_{2j-1}$ and $R_{2j}$.
    Similarly, if Breaker places a point vertically below a point $s_{2j-1,1}$ or $s_{2j,k-1}$ for some $j\in [2t]$, then we replace this point with Breaker's point.
    In particular, if any of Breaker's points lies on the boundary of a region $R_i$, $i\in [4t]$, then we use it as a replacement.
    Any such replacement can only decrease the regions $R_{2j-1}$ and $R_{2j}$, and it maintains the properties that both $S_{2j-1}\cup \{p_j\}$ and $S_{2j}\cup \{p_j\}$ are caps, and the regions $R_i$, $i\in [4t]$, have pairwise disjoint interiors.

    By the pigeonhole principle, there are at most $t$ indices $j\in [2t]$ such that Breaker
    placed two or more points in the interior of $R_{2j-1}\cup R_{2j}$.
     Consequently, there are at least $t$ indices $j\in [2t]$ such that Breaker
    placed at most one point in the interior of $R_{2j-1}\cup R_{2j}$.
    For any such pair, the interior of $R_{2j-1}$ or $R_{2j}$ is empty,
    and so $S_{2j-1}\cup \{p_j\}$ or $S_{2j}\cup \{p_j\}$ forms a $k$-strip
    with respect to a downward halfline. Therefore, we find at least $t$ parallel $k$-strips among the points placed by both players.

   \paragraph{Analyzing the number of rounds.} Let $r(k,t)$ denote the minimum number of rounds such that Maker can ensure that at the end of round $r(k,t)$, there are $t$ parallel $k$-strips among the $2r(k,t)$ points placed by both players. As argued above, we have $r(1,t)=t/2$ and $r(2,t)=t$ for all $t\in \mathbb{N}$; and we proved the recurrence relation $r(k,t)=r(k-1,4t)+2t$ for all $k\geq 3$ and $t\in \mathbb{N}$. The recursion solves to $r(k,t)=
    (\frac53\cdot 4^{k-2}-\frac23)t$ for all $k\geq 2$.
\end{proof}

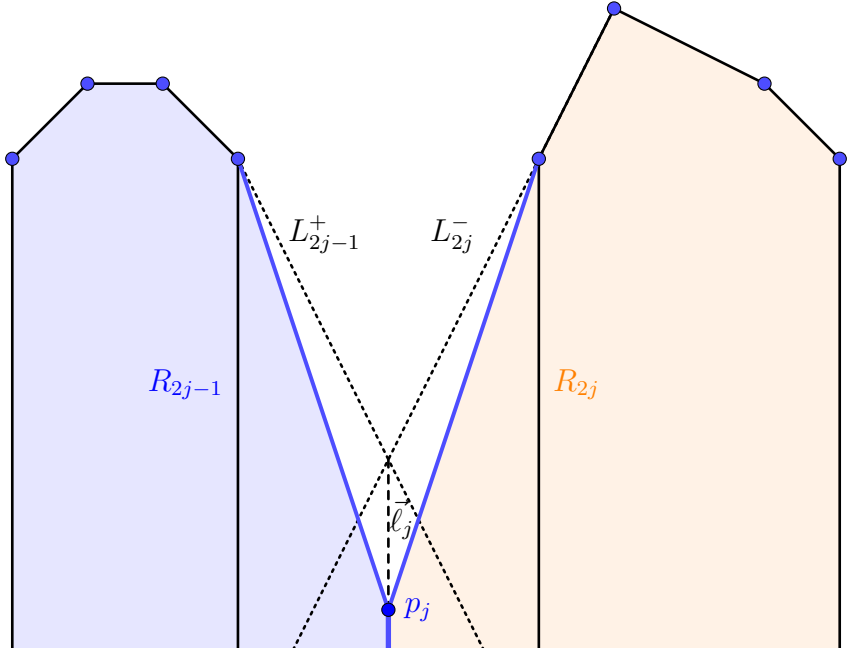
\begin{figure}
    \centering
    \begin{tikzpicture}[line cap=round,line join=round,>=triangle 45,x=1cm,y=1cm]
\clip(-5.5,-3.5) rectangle (6.5,5.5);
\fill[line width=1pt,color=qqqqff,fill=qqqqff,fill opacity=0.1] (0,-3) -- (-2,3) -- (-3,4) -- (-4,4) -- (-5,3) -- (-5,-10) -- (0,-10) -- cycle;
\fill[line width=1pt,color=qqqqff,fill=orange,fill opacity=0.1] (0,-3) -- (2,3) -- (3,5) -- (5,4) -- (6,3) -- (6,-10) -- (0,-10) -- cycle;
\draw [line width=1pt] (-5,3)-- (-4,4);
\draw [line width=1pt] (-4,4)-- (-3,4);
\draw [line width=1pt] (-3,4)-- (-2,3);
\draw [line width=1pt] (2,3)-- (3,5);
\draw [line width=1pt] (5,4)-- (3,5);
\draw [line width=1pt] (6,3)-- (5,4);
\draw [line width=1pt] (-5,3) -- (-5,-9.9);
\draw [line width=1pt] (-2,3) -- (-2,-9.9);
\draw [line width=1pt] (2,3) -- (2,-9.9);
\draw [line width=1pt] (6,3) -- (6,-9.9);
\draw [line width=1pt,dotted,domain=-12.3:3] plot(\x,{(--1-2*\x)/-1});
\draw [line width=1pt,dotted,domain=-2:13.3] plot(\x,{(-1-2*\x)/1});
\draw [line width=1pt,dashed] (0,-1) -- (0,-9.9);
\draw [line width=1.5pt, color=ududff] (-2,3) -- (0,-3);
\draw [line width=1.5pt, color=ududff] (2,3) -- (0,-3);
\draw [line width=2pt, color=ududff] (0,-5) -- (0,-3);
\begin{scriptsize}
\draw [fill=ududff] (-5,3) circle (2.5pt);
\draw [fill=ududff] (-4,4) circle (2.5pt);
\draw [fill=ududff] (-3,4) circle (2.5pt);
\draw [fill=ududff] (-2,3) circle (2.5pt);
\draw [fill=ududff] (2,3) circle (2.5pt);
\draw [fill=ududff] (3,5) circle (2.5pt);
\draw [fill=ududff] (5,4) circle (2.5pt);
\draw [fill=ududff] (6,3) circle (2.5pt);
\draw[color=black] (0.85,2) node {\large $L_{2j}^-$};
\draw[color=black] (-0.85,2) node {\large $L_{2j-1}^+$};
\draw[color=black] (0.18,-1.8) node {\large $\vec{\ell}_j$};
\draw [fill=qqqqff] (0,-3) circle (2.5pt);
\draw[color=qqqqff] (0.4,-3) node {\large $p_j$};
\draw[color=qqqqff] (-2.7,0) node {\large $R_{2j-1}$};
\draw[color=orange] (2.5,0) node {\large $R_{2j}$};
\end{scriptsize}
\end{tikzpicture}

\caption{Placing a point $p_j$ on the halfline $\vec{\ell}_j$, Maker creates two candidates for possible $k$-strips.\label{fig:k-strips}}
\end{figure}

\subsection{Monochromatic Maker-Breaker Game with Bias}

We next consider a variation of the Maker-Breaker game in which Breaker has the advantage that she is allowed to place more points than Maker in each round. Nevertheless, Maker still has a winning strategy.

\begin{theorem}\label{thm:MB}
Consider the monochromatic Maker-Breaker game with $1:s$ bias.
For every $k,t \in \mathbb{N}$ and $s\geq 1$,  Maker can ensure that there are $t$ parallel $k$-strips among the points placed by both players. 
    This can be accomplished in at most $t/(s+1)$ rounds for $k=1$ and at most $((4s)^{k-2}/(s+1)+((4s)^{k-2}-1)/(4s-1))2t$ rounds for $k\geq 2$ when $s\in \mathbb{N}$.
\end{theorem}

\begin{proof}
    Without loss of generality, assume that $s$ is an integer. We proceed by induction on $k$. As before, the base cases $k=1$ and $k=2$ are trivial. Assume next that we are given $k,t\in \mathbb{N}$, $k\geq 2$, and the theorem holds for all $k',t'\in \mathbb{N}$ with $k'<k$. By the induction hypothesis, Maker can create $4st$ parallel $(k-1)$-strips with respect to the current point set $S$. As previously, we may assume (by rotating the coordinate system if necessary) all $k$-strips are w.r.t. a downward halfline $\vec{d}=\overrightarrow{(0,-1)}$, and all points in $S$ have distinct $x$-coordinates. Denote the $(k-1)$-strips by $S_1, S_2, \ldots ,S_{4st}$, from left to right, and the points in them as $s_{i,j}$, again ordered according to increasing $x$-coordinate. For each $(k-1)$-strip $S_i$, we define the regions $Q_i=\conv(S_i)+\vec{d}$ (see \Cref{fig: Bichrom groups} for a illustration). Note that the regions $Q_i$ are pairwise disjoint.

    We partition the $4st$ parallel $(k-1)$-strips into $2t$ groups $\mathcal{S}_1, \dots, \mathcal{S}_{2t}$, each consisting of $2s$ consecutive $(k-1)$-strips.
    For each group $\mathcal{S}_j$, we also define the region $\mathcal{Q}_j=\bigcup\{Q_i: S_i\in \mathcal{S}_j\}$, that is, the union of the regions $Q_i$ of the corresponding $(k-1)$-strips. Note that the regions $\mathcal{Q}_j$ are pairwise disjoint.
    For each group $\mathcal{S}_j$, $j\in [2t]$, let $\ell_j$ be a vertical line that lies to the left of $\mathcal{S}_j$, but right of any other $\mathcal{S}_{j'}$, $j'<j$. Let
    \[
    \vec{\ell}_j= \ell_j\cap \left( \bigcap_{S_i\in \mathcal{S}_j} C_i^-\right) ,
    \]
    and note that $\vec{\ell}_j$ is a downward halfline. Also, for any point $p_j\in \vec{\ell}_j$ and any $(k-1)$-strip $S_i\in \mathcal{S}_j$, the set $S_i\cup \{p_j\}$ is a $k$-strip with respect to the point set $S\cup \{p_j\}$; see \Cref{fig: Bichrom groups}.

    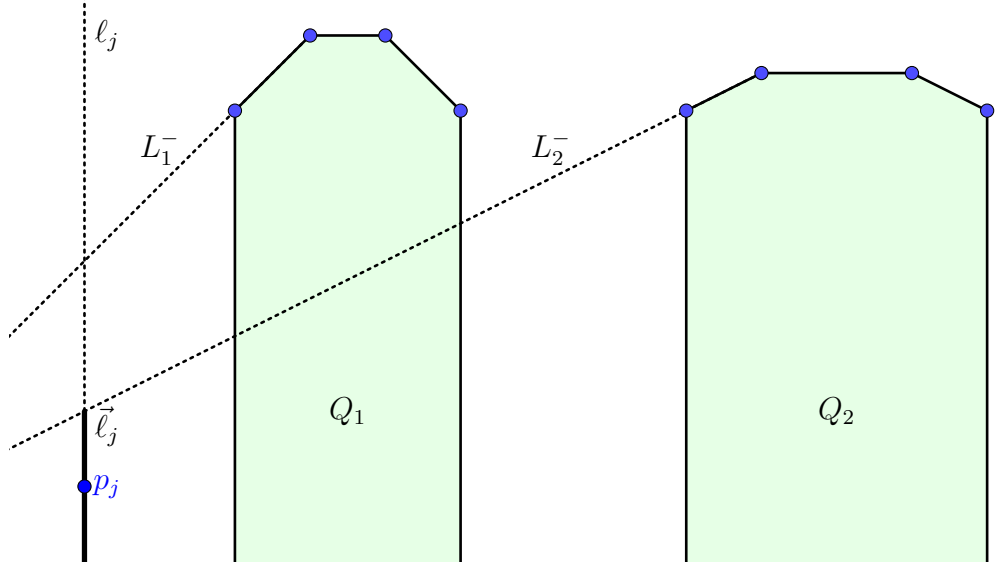
\begin{figure}
        \centering
        \begin{tikzpicture}[line cap=round,line join=round,>=triangle 45,x=1cm,y=1cm]
\clip(-7,-3) rectangle (6.5,4.5);
\fill[line width=1pt,color=qqqqff,fill=green,fill opacity=0.1] (-4,3) -- (-3,4) -- (-2,4) -- (-1,3) -- (-1,-4) -- (-4,-4) -- cycle;
\fill[line width=1pt,color=qqqqff,fill=green,fill opacity=0.1] (2,3) -- (3,3.5) -- (5,3.5) -- (6,3) -- (6,-4) -- (2,-4) -- cycle;
\draw [line width=1pt] (-4,3)-- (-3,4);
\draw [line width=1pt] (-2,4)-- (-3,4);
\draw [line width=1pt] (-2,4)-- (-1,3);
\draw [line width=1pt] (2,3)-- (3,3.5);
\draw [line width=1pt] (3,3.5)-- (5,3.5);
\draw [line width=1pt] (5,3.5)-- (6,3);
\draw [line width=1pt] (-4,3) -- (-4,-6.9);
\draw [line width=1pt] (-1,3) -- (-1,-6.9);
\draw [line width=1pt] (2,3) -- (2,-6.9);
\draw [line width=1pt] (6,3) -- (6,-6.9);
\draw [line width=1pt,dotted] (-6,-6.9) -- (-6,5.4);
\draw [line width=1pt,dotted,domain=-9.4:3] plot(\x,{(-2-0.5*\x)/-1});
\draw [line width=1pt,dotted,domain=-9.4:-3] plot(\x,{(-7-1*\x)/-1});
\draw [line width=2pt] (-6,-1) -- (-6,-3.5);
\begin{scriptsize}
\draw [fill=ududff] (-4,3) circle (2.5pt);
\draw [fill=ududff] (-3,4) circle (2.5pt);
\draw [fill=ududff] (-2,4) circle (2.5pt);
\draw [fill=ududff] (-1,3) circle (2.5pt);
\draw [fill=ududff] (2,3) circle (2.5pt);
\draw [fill=ududff] (3,3.5) circle (2.5pt);
\draw [fill=ududff] (5,3.5) circle (2.5pt);
\draw [fill=ududff] (6,3) circle (2.5pt);
\draw[color=black] (-2.5,-1) node {\large $Q_1$};
\draw[color=black] (4,-1) node {\large $Q_2$};
\draw[color=black] (0.2,2.5) node {\large $L_2^-$};
\draw[color=black] (-5,2.5) node {\large $L_1^-$};
\draw[color=black] (-5.7,4) node {\large $\ell_j$};
\draw[color=black] (-5.7,-1.2) node {\large $\vec{\ell}_j$};
\draw [fill=qqqqff] (-6,-2) circle (2.5pt);
\draw[color=qqqqff] (-5.7,-2) node {\large $p_j$};
\end{scriptsize}
\end{tikzpicture}

    \caption{Example of a group of $4$-strips for $s=1$. The point $p_j$ creates two possible $k$-strips.}
        \label{fig: Bichrom groups}
    \end{figure}

    In the next $2t$ rounds, Maker places arbitrary points $p_j\in \vec{\ell}_j$ for all $j\in [2t]$ such that the set of all points placed so far is in general position.
    For every $j\in [2t]$, point $p_j$ creates $2s$ \emph{candidates} for possible $k$-strips:
    $S_i\cup \{p_j\}$ for all $S_i\in \mathcal{S}_j$.

    In the meantime, Breaker plays $2st$ points.
    By the pigeonhole principle, there are at most $t$ indices $j\in [2t]$ such that Breaker places $2s$ or more points in the region $\mathcal{Q}_j$. Consequently, there are at least $t$ indices $j\in [2t]$ such that Breaker placed fewer than $2s$ points in $\mathcal{Q}_j$.
    We show that each such group yields at least one $k$-strip.

    Consider a group $\mathcal{S}_j$ such that Breaker placed at most $2s-1$ points in $\mathcal{Q}_j$. By the pigeonhole principle, there is a $(k-1)$-strip $S_i\in \mathcal{S}_j$ such that Breaker placed no points in the region $Q_i$.
    Consider the cone $C_i^-$. If Breaker did not place any point in $C_i^-$,
    then $S_i\cup \{p_j\}$ is a $k$-strip at this time.
    Otherwise, let $C_i^*$ be the cone swept by rotating the halfline $L(s_{i,1})$ clockwise about $s_{i,1}$ until it passes through one of Breaker's points: let $q$ be the first such point.
    Since Breaker placed a point in $C_i^-$, then $C_i^*\subset C_i^-$.
    Consequently, $S_i\cup \{q\}$ is a $k$-strip at this time.

So in each group $\mathcal{S}_i$, there is a $k$-strip, adding to a total of~$t$ strips, as required.

\paragraph{Analyzing the number of rounds.} Let $r(k,t)$ denote the minimum number of rounds such that Maker can ensure that at the end of round $r(k,t)$, there are $t$ parallel $k$-strips among the $(s+1)r(k,t)$ points placed by both players. Clearly, we have $r(1,t)=t/(s+1)$ and $r(2,t)=2t/(s+1)$ for all $t\in \mathbb{N}$; and we proved the recurrence relation $r(k,t)=r(k-1,4st)+2t$ for all $k\geq 3$ and $t\in \mathbb{N}$. The recursion solves to $r(k,t)=((4s)^{k-2}/(s+1)+((4s)^{k-2}-1)/(4s-1))2t$ for $k\geq 2$ when $s$ is a positive integer.
\end{proof}

\section{Bichromatic Maker-Breaker Game with Bias}
\label{sec:bichromatic}

In this section, we consider the bichromatic Maker-Breaker game with $1:s$ bias, first with $0\leq s<2$ and later for $s \geqslant 2$. We show that in the first case the game is a win for Maker while for sufficiently large $s$ (depending on $k$) the game is a win for Breaker.

\begin{theorem}\label{thm:s2}
The bichromatic Maker-Breaker game with $1:s$ bias, for any constant $0\leq s<2$ and for any integer $k\geq 3$, is a win for Maker.
\end{theorem}

We start by making a straightforward observation about how strips can be slightly changed in direction to avoid having a point in their interiors. This is now necessary since earlier, Breaker's points could also be used to extend a strip, while now this is no longer the case. Moreover, given several disjoint downward strips, rotate the downward halfline so that they remain parallel w.r.t.\ another direction.
\begin{observation}\label{Obs_rotatestrip}
Let $\conv(P)+\vec{d}$ and $\conv(Q)+\vec{d}$ be downward strips disjoint from a point set $S$. Then there exists a $\delta_0>0$ such that for all $\delta\in (-\delta_0, \delta_0)$ and the halfline $\vec{v}=\overrightarrow{(\delta, -1)}$, the following hold:
\begin{itemize}
\item both $\conv(P)+\vec{v}$ and $\conv(Q)+\vec{v}$ are strips disjoint from $S$;
\item if $(\conv(P)+\vec{d}) \cap (\conv(Q)+\vec{d}) = \emptyset$, then $(\conv(P)+\vec{v}) \cap (\conv(Q)+\vec{v}) = \emptyset$;
\item if $\conv(P)\cap \conv(Q)=\{s\}$ and $(\conv(P)+\vec{d})\cap  (\conv(Q)+\vec{d})= \{s\}+\vec{d}$, then $(\conv(P)+\vec{v})\cap (\conv(Q)+\vec{v})=\{s\}+\vec{v}$.
\end{itemize}
\end{observation}
\noindent
Now we can prove the theorem.

\begin{proof}[Proof of \Cref{thm:s2}]
Let $\varepsilon=2-s$, that is, we assume $1:(2-\varepsilon)$ bias in favor of Breaker.
We prove, by induction on $k$, that for all $k,r\in \mathbb{N}$, Maker can build $r$ parallel $k$-strips disjoint from Breaker's points. The initial cases $k\in\{1,2\}$ are easy, so we assume that $k\geq 3$ and Maker has built $2t$ parallel $(k-1)$-strips, where $t:= \lceil r/\varepsilon\rceil$.
As before, without loss of generality, we may assume that these are downward strips $S_1, \dots, S_{2t}$, labeled in increasing order by $x$-projections of their convex hulls. Let $S$ be the set of points played so far. We label the points in $S_i$ by $s_{i,1}, \dots, s_{i,k-1}$.

For each $j \in [2t-1]$, let $\vec{\ell}_j$ be a downward pointing halfline between $S_j$ and $S_{j+1}$, similarly defined as above. Any point $p_j \in \vec{\ell}_j$ has the property that $S_j \cup \{p_j\}$ and $S_{j+1} \cup \{p_j\}$ are also downward strips disjoint from $S$.
Note that for $i\neq j$, the strips $S_i \cup \{p_i\}$ and $S_{i+1} \cup \{p_i\}$ are disjoint from the strips $S_j \cup \{p_j\}$ and $S_{j+1} \cup \{p_j\}$.

By \Cref{Obs_rotatestrip}, we find $\delta_0>0$ so that for all $\delta\in (-\delta_0,\delta_0)$ and for halfline $\vec{v}=\overrightarrow{(\delta,-1)}$, the regions $\conv(S_i\cup \{p_i\})+\vec{v}$ and $\conv(S_{i+1}\cup \{p_i\})+\vec{v}$ satisfy the following:
\begin{enumerate}
\item both $\conv(S_i\cup \{p_i\})+\vec{v}$ and $\conv(S_{i+1} \cup \{p_i\})+\vec{v}$ are strips;
\item $\conv(S_i\cup \{p_i\})+\vec{v}$ and $\conv(S_{i+1}\cup \{p_i\})+\vec{v}$ intersect only in the line $p_i+\vec{v}$;
\item for $i\neq j$, the strips $\conv(S_i\cup \{p_i\})+\vec{v}$ and $\conv(S_{i+1} \cup \{p_i\})+\vec{v}$ are disjoint from the strips $\conv(S_j \cup \{p_j\})+\vec{v}$ and $\conv(S_{j+1} \cup \{p_j\})+\vec{v}$; and
\item the strips $\conv(S_i\cup \{p_i\})+\vec{v}$ and $\conv(S_{i+1}\cup \{p_i\})+\vec{v}$ are disjoint from $S$.
\end{enumerate}

Maker plays the points $p_1, \dots, p_t$. During these moves Breaker plays
points $b_1, \ldots, b_{(2-\varepsilon)t}$.
Pick a particular $\delta\in (-\delta_0, \delta_0)$ such that none of $b_1, \ldots, b_{(2-\varepsilon)t}$ is on any line in direction $(\delta,-1)$ through $p_1, \ldots, p_t$; this can be done because there are finitely many points to avoid, but infinitely many choices for $\vec{v}=\overrightarrow{(\delta,-1)}$. This, in particular, ensures that each $b_i$ is in at most one of the regions $\conv(S_i\cup \{p_i\})+\vec{v}$, $\conv(S_{i+1}\cup \{p_i\})+\vec{v}$, for $i=1, \dots, t$ (using (2) and (3)). By the pigeonhole principle, we can find at least $\varepsilon t = \varepsilon\cdot \lceil r/\varepsilon\rceil\geq r$ disjoint strips whose downward regions do not contain any Breaker points, as required.
\end{proof}

Note that for $s=2$, Maker's strategy that we described no longer works. Indeed, in the inductive step, Breaker can place points to the left and right of each of the $t$ points placed by Maker (with respect to the projection on $x$-coordinate) making sure that each $k$-strip  contains one of her points either in the interior or as a vertex; see \Cref{fig:bicromatic k-strips}. We shall prove below that for $k \geqslant 8$, Breaker can win with speed $s\geq 12$, but it remains an open problem to determine the minimum speed that allows Breaker to win.

    \begin{figure}
        \centering
        \begin{tikzpicture}[line cap=round,line join=round,>=triangle 45,x=0.75cm,y=0.75cm]
\clip(-1,-3) rectangle (13,7);
\fill[line width=1pt,color=qqqqff,fill=green,fill opacity=0.1] (1,4) -- (1.44,5) -- (2.4,5.18) -- (4,4) -- (4,-3) -- (1,-3) -- cycle;
\fill[line width=1pt,color=qqqqff,fill=green,fill opacity=0.1] (8,4.5) -- (9,6) -- (10,6) -- (11,5) -- (11,-3) -- (8,-3) -- cycle;
\draw [line width=1pt,dotted] (1,4) -- (1,-5.626666666666647);
\draw [line width=1pt,dotted] (4,4) -- (4,-5.626666666666647);
\draw [line width=1pt] (1,4)-- (1.44,5);
\draw [line width=1pt] (1.44,5)-- (2.4,5.18);
\draw [line width=1pt] (2.4,5.18)-- (4,4);
\draw [line width=1pt] (1.44,5)-- (1,4);
\draw [line width=1pt] (8,4.5)-- (9,6);
\draw [line width=1pt] (9,6)-- (10,6);
\draw [line width=1pt] (10,6)-- (11,5);
\draw [line width=1pt,dotted] (8,4.5) -- (8,-5.626666666666647);
\draw [line width=1pt,dotted] (11,5) -- (11,-5.626666666666647);
\draw [line width=1pt,dashed ,color=qqttcc] (4,4)-- (6,-2);
\draw [line width=1pt,dashed,color=qqttcc] (8,4.5)-- (6,-2);
\begin{scriptsize}
\draw [fill=ududff] (1,4) circle (2.5pt);
\draw [fill=ududff] (1.44,5) circle (2.5pt);
\draw [fill=ududff] (2.4,5.18) circle (2.5pt);
\draw[color=ududff] (2,5.8) node {$S_j$};
\draw [fill=ududff] (4,4) circle (2.5pt);
\draw [fill=ududff] (8,4.5) circle (2.5pt);
\draw [fill=ududff] (9,6) circle (2.5pt);
\draw[color=ududff] (9.5,6.5) node {$S_{j+1}$};
\draw [fill=ududff] (10,6) circle (2.5pt);
\draw [fill=ududff] (11,5) circle (2.5pt);
\draw [fill=ffwwqq] (5.46,2) circle (2.5pt);
\draw [fill=ffwwqq] (0.22,3.36) circle (2.5pt);
\draw [fill=ffwwqq] (6.7,3.94) circle (2.5pt);
\draw [fill=ududff] (6,-2) circle (2.5pt);
\draw[color=ududff] (6,-2.4) node {$p_j$};
\draw [fill=ffwwqq] (5.5,-2) circle (2.5pt);
\draw [fill=ffwwqq] (6.5,-2) circle (2.5pt);
\end{scriptsize}
\end{tikzpicture}
\caption{If Breaker has a bias $s\geq 2$ in the bichromatic game, Maker's strategy to create $k$-strips no longer works. Breaker can place one point left and right of $p_j$. }
\label{fig:bicromatic k-strips}
\end{figure}
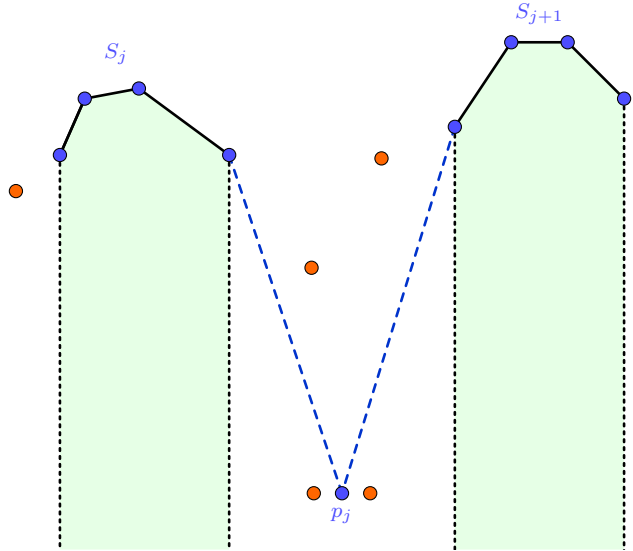

\subsection{One-Round Bichromatic Game with Bias in Favor of Breaker}
\label{subsec:Breaker}

In the remainder of this section, we consider the case with $s\geq 2$.

\begin{proposition}\label{obs:bichrom}
For every $k,n\geq 3$, Breaker wins in the bichromatic one-round Maker-Breaker game with bias $1:2$.
\end{proposition}

\begin{proof}
Suppose Maker plays a set $M$ of $n$ points in the plane in general position. Assume (by rotating the coordinate system if necessary) that no two points in $M$ have the same $x$-coordinate. For a winning strategy, Breaker plays as follows. Let $\delta$ be the minimum distance between a point in $M$ and a line spanned by two other points in $M$; and let $L\geq 2$ be an upper bound on the absolute values of all slopes of lines spanned by two points in $M$. For each point $p\in M$ with coordinates $p=(p_1,p_2)$, Breaker places two points: $(p_1,p_2-\delta)$ and $(p_1+\delta/L, p_2+\delta/2)$.

To show that Breaker wins for all $k\geq 3$, it is enough to show that $B$ contains a point in the interior of every triangle spanned by $M$. Let $\Delta abc$ be a triangle spanned by $M$.
By our assumption, we may assume that $a_1<b_1<c_1$, and suppose that point $b$ lies below the line spanned by $a$ and $c$; refer to \Cref{fig:2-round bichromatic case}. By the choice of $\delta$, point $(b_1+\delta/L,b_2+\delta/2) \in B$ is also below the line spanned by $a$ and $c$, any by the choice of $L$, it is in the interior of $\Delta abc$.  If $b$ lies above the line spanned by $a$ and $c$, then the point $(b_1,b_2-\delta/2) \in B$ is at distance less than $\delta$ from $b$, and so it is in the interior of $\Delta abc$.
\end{proof}

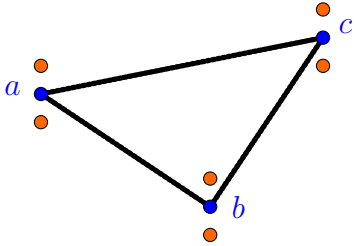
\begin{figure}
\centering
\begin{tikzpicture}[line cap=round,line join=round,>=triangle 45,x=0.75cm,y=0.75cm]
\clip(-3,-3) rectangle (4,4);
\draw [line width=2pt,dash pattern=on 1pt off 1pt] (-2,1)-- (3,2);
\draw [line width=2pt,dash pattern=on 1pt off 1pt] (3,2)-- (1,-1);
\draw [line width=2pt,dash pattern=on 1pt off 1pt] (1,-1)-- (-2,1);
\begin{scriptsize}
\draw [fill=blue] (1,-1) circle (2.5pt);
\draw[color=blue] (1.5,-1) node {\large $b$};
\draw [fill=blue] (-2,1) circle (2.5pt);
\draw[color=blue] (-2.5,1.1) node {\large $a$};
\draw [fill=blue] (3,2) circle (2.5pt);
\draw[color=blue] (3.4,2.2) node {\large $c$};
\draw [fill=ffwwqq] (-2,1.5) circle (2.5pt);
\draw [fill=ffwwqq] (-2,0.5) circle (2.5pt);
\draw [fill=ffwwqq] (1,-0.5) circle (2.5pt);
\draw [fill=ffwwqq] (1,-1.5) circle (2.5pt);
\draw [fill=ffwwqq] (3,2.5) circle (2.5pt);
\draw [fill=ffwwqq] (3,1.5) circle (2.5pt);
\end{scriptsize}
\end{tikzpicture}
    \caption{In a 1-round bichromatic game, Breaker can always ensure that no triangle spanned by Maker's points forms a 3-hole. Maker's points are blue and Breaker's points are orange.}
    \label{fig:2-round bichromatic case}
\end{figure}

\subsection{Multi-Round Bichromatic Game with Bias in Favor of Breaker}
\label{ssec:miltiround}

In this subsection, we address the multi-round bichromatic Maker-Breaker game, and show that Breaker wins with sufficient bias $s$ in her favor. Breaker follows the basic strategy in \Cref{obs:bichrom}, that is, placing points in some $\delta$-neighborhood of Maker's points. However, the suitable value of $\delta$ may decrease as Maker adds new points. In the algorithm below, Breaker chooses $\delta_i>0$ based on the first $i$ points played by Maker. In round~$i$, Breaker places $q$ points in a $\delta_i$-neighborhood of Maker's point and can choose up to one previous point $p_j$, $j<i$, played by Maker, and places new points in the smaller $\delta_i$-neighborhood of $p_j$.

\paragraph{Breaker's strategy for $s=2\lambda$.} Let $M_i=\{p_1,\ldots , p_i\}$ denote the first $i$ points placed by Maker.
Breaker maintains a set $\mathcal{D}_i=\{D_1,\ldots , D_i\}$ of pairwise disjoint closed disks, where $D_j$ is centered at $p_j$ for all $j\in \{1,\ldots , i\}$. In round $i$, Maker first plays a new point $p_i$, and then Breaker plays up to $2\lambda$ points as follows. Breaker chooses two real numbers $0<r_i<R_i$ (specified below), then creates a disk $D_i$ of radius $R_i$ and a circle $C_i$ of radius $r_i$ centered at $p_i$, and places $\lambda$ points at the vertices of a regular $\lambda$-gon inscribed in $C_i$; followed by a perturbation that moves the points on the circle $C_i$ to ensure that no three points are collinear. Specifically, when $\lambda$ is odd, rotate the regular $\lambda$-gon such that none of its vertices is on a line spanned by $p_i$ and a previous point. However, if $\lambda$ is even, then $p_i$ is the midpoint of the main diagonals of the regular $\lambda$-gon and we need to perturb the points more carefully.

If $\lambda$ is even, we choose $\lambda$ points on the circle $C_i$ as follows. In round~1 (i.e., $i=1)$, we can place the $\lambda$ points on $C_1$ arbitrarily in general position. For $i\geq 2$, let $\varepsilon_i\in (0,\frac{\pi}{10\lambda})$ be less than half of the smallest nonzero angle in $\{\angle p_x p_i p_y \mod \frac{2\pi}{\lambda}: 1\leq x,y\leq i\}$. Initially let $(q_1^0,\ldots , q_\lambda^0)$ be the vertices of a regular $\lambda$-gon inscribed in $C_i$ such that $\angle q_j^0 p_i p_{i'}> \varepsilon_i$ for all $1\leq j\leq \lambda$ and $1\leq i'<i$.
Note that such a $\lambda$-gon always exists, as each $1\leq i'<i$ can rule out a rotation range of size at most $2\varepsilon_i$ for each $1\leq j\leq \lambda$, in total at most $2\lambda \frac{\pi}{10\lambda}<\frac{\pi}{5}<2\pi$.
See \Cref{fig:perturbation} for an illustration. For $j=1,\ldots , \lambda$, we perturb $q_j^0$ to points $q_j$ as follows: rotate $q_j^0$ along $C_i$ counterclockwise by angle $\frac{1}{j}\varepsilon_i$.
(This is to ensure that $(\star)$ holds---see below.)
Note that the segments $q_j^0q_j$ do not cross any line $p_ip_{i'}$, $i'<i$, and the perturbation guarantees $\angle q_{j} p_i q_{j+1} \leq 2\pi/\lambda$ for $j=1,\ldots ,\lambda-1$, and $2\pi/\lambda<\angle q_\lambda p_i q_1 < 2\pi/\lambda +\varepsilon_i$.

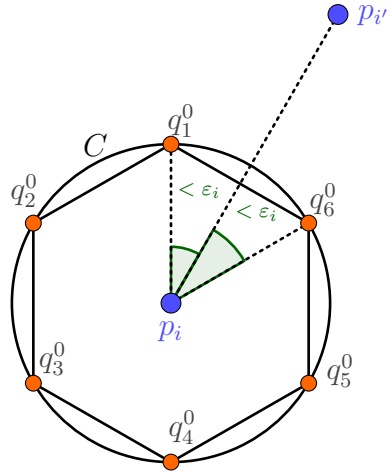
\begin{figure}
    \centering
    \begin{tikzpicture}[line cap=round,line join=round,>=triangle 45,x=1cm,y=1cm, scale=1.5]
\clip(-2,-2) rectangle (2,3);
\draw [shift={(0,0)},line width=1pt,color=qqwuqq,fill=qqwuqq,fill opacity=0.1] (0,0) -- (60:0.5) arc (60:90:0.5) -- cycle;
\draw [shift={(0,0)},line width=1pt,color=qqwuqq,fill=qqwuqq,fill opacity=0.1] (0,0) -- (30:0.75) arc (30:60:0.75) -- cycle;
\draw [line width=1pt] (0,0) circle (1.41cm);
\draw [line width=1pt,dotted] (1.48,2.56)-- (0,0);
\draw [line width=1pt, dotted] (0,1.41)-- (0,0);
\draw [line width=1pt, dotted] (0,0)-- (1.22,0.71);
\draw [line width=1pt] (0,1.41)-- (-1.22,0.71);
\draw [line width=1pt] (-1.22,0.71)-- (-1.22,-0.71);
\draw [line width=1pt] (-1.22,-0.71)-- (0,-1.41);
\draw [line width=1pt] (0,-1.41)-- (1.22,-0.71);
\draw [line width=1pt] (1.22,-0.71)-- (1.22,0.71);
\draw [line width=1pt] (1.22,0.71)-- (0,1.41);
\begin{scriptsize}
\draw [fill=ududff] (0,0) circle (2.5pt);
\draw[color=ududff] (0,-0.25) node {\large $p_{i}$};
\draw[color=black] (-0.67,1.4) node {\large $C$};
\draw [fill=ududff] (1.48,2.56) circle (2.5pt);
\draw[color=ududff] (1.8,2.55) node {\large $p_{i'}$};
\draw [fill=ffwwqq] (0,1.41) circle (2pt);
\draw[color=uuuuuu] (0.09,1.62) node {\large $q_1^0$};
\draw [fill=ffwwqq] (-1.22,0.71) circle (2pt);
\draw[color=uuuuuu] (-1.3,1) node {\large $q_2^0$};
\draw [fill=ffwwqq] (-1.22,-0.71) circle (2pt);
\draw[color=uuuuuu] (-1.05,-0.49) node {\large $q_3^0$};
\draw [fill=ffwwqq] (0,-1.41) circle (2pt);
\draw[color=uuuuuu] (0.1,-1.10) node {\large $q_4^0$};
\draw [fill=ffwwqq] (1.22,-0.71) circle (2pt);
\draw[color=uuuuuu] (1.5,-0.6) node {\large $q_5^0$};
\draw [fill=ffwwqq] (1.22,0.71) circle (2pt);
\draw[color=uuuuuu] (1.35,0.92) node {\large $q_6^0$};
\draw[color=qqwuqq] (0.76,0.8) node {$> \varepsilon_i $};
\draw[color=qqwuqq] (0.25,1) node {$> \varepsilon_i$};
\end{scriptsize}
\end{tikzpicture}
    \caption{Illustration on how to set up the to-be-perturbed $\lambda$-gon for $\lambda = 6$ and one point $p_{i'}$. We choose the points $(q_1^0, \dots, q_\lambda^0)$ such that they are the vertices of a regular 6-gon and such that $\angle q_j^0 p_i p_{i'} > \varepsilon_i$ for all $j \in [\lambda]$. For the subsequent perturbation, we move $q_j^0$ counterclockwise along $C$ by angle $\varepsilon_i / j$ for all $j=1,\ldots, \lambda$.}
    \label{fig:perturbation}
\end{figure}

Furthermore, if $p_i\in D_j$ for some $j<i$, then Breaker chooses new values for $r_j$ and $R_j$ such that $0<r_j<R_j$, and replaces the disk $D_j$ and the circle $C_j$; and places $\lambda$ new points on the new circle $C_j$ (the new disk and circle will still be denoted by $D_j$ and $C_j$, respectively). Overall, Breaker places at most $s=2\lambda$ points in round $i$.

It remains to specify the radii $0<r_i<R_i$, and the possible replacement radii $0<r_j<R_j$. We first specify $R_i$ and $R_j$. If $p_i\notin D_j$ for any $j<i$, then let $R_i$ be any positive real such that $D_i$ is disjoint from all previous disks $D_j$, $j<i$. Otherwise, $p_i\in D_j$ for a unique $j\in \{1,\ldots ,i-1\}$. In this case, we set $R_j:=|p_ip_j|/2$, so that the new disk $D_j$ of radius $R_j$ centered at $p_j$ does not contain $p_i$, and then let $R_i>0$ be any real such that $D_i$ is disjoint from all previous disks $D_j$, $j<i$.

For choosing the radius $r_i$ (and $r_j$ if necessary), we proceed as follows. For all $a\in \{p_1,\ldots , p_{i-1}\}$, let $c_i(a)$ be the intersection point of the line $ap_i$ and the boundary of $D_i$ such that $\angle ap_ic_i(a)=\pi$;
see \Cref{fig: defining ci}. Now let $r_i>0$ be sufficiently small so that the circle of radius $r_i$ centered at $p_i$ is disjoint from the lines
$b c_i(a)$  
and $ab$ for all $\{a,b\}\subset \{p_1,\ldots, p_{i-1}\}$.
In case the disk $D_j$ has been replaced, then we recompute the points $c_j(a)$ for all $a\in \{p_1,\ldots, p_{j-1}\}$.
This completes the description of Breaker's strategy for $s=2\lambda$.

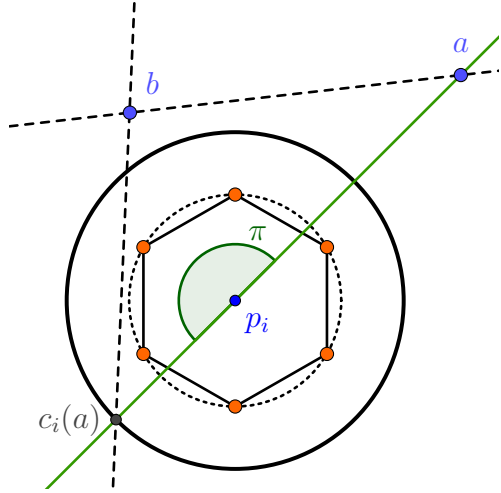
\begin{figure}
    \centering
    \begin{tikzpicture}[line cap=round,line join=round,>=triangle 45,x=1cm,y=1cm, scale=1]
\clip(-3,-2.5) rectangle (3.5,4);
\draw [shift={(0,0)},line width=1pt,color=qqwuqq,fill=qqwuqq,fill opacity=0.1] (0,0) -- (45:0.75) arc (45:225:0.75) -- cycle;
\draw [line width=1.5pt] (0,0) circle (2.24cm);
\draw [line width=1pt,dotted] (0,0) circle (1.41cm);
\draw [line width=1pt] (0,1.41)-- (1.22,0.71);
\draw [line width=1pt] (1.22,0.71)-- (1.22,-0.71);
\draw [line width=1pt] (1.22,-0.71)-- (0,-1.41);
\draw [line width=1pt] (0,-1.41)-- (-1.22,-0.71);
\draw [line width=1pt] (-1.22,-0.71)-- (-1.22,0.71);
\draw [line width=1pt] (-1.22,0.71)-- (0,1.41);
\draw [line width=1pt, dashed, domain=-3:1] plot(\x,{22.66*\x + 34.33});
\draw [line width=1pt, dashed, domain=-3:4] plot(\x,{0.113636*\x + 2.65909});
\draw [line width=1pt, color=ttzzqq,domain=-2.97:6.20] plot(\x,\x);
\begin{scriptsize}
\draw [fill=qqqqff] (0,0) circle (2pt);
\draw[color=qqqqff] (0.3,-0.3) node {\large $p_i$};
\draw [fill=ududff] (3,3) circle (2.5pt);
\draw [fill=ududff] (-1.4,2.5) circle (2.5pt);
\draw[color=ududff] (-1.1,2.9) node {\large $b$};
\draw[color=ududff] (3,3.4) node {\large $a$};
\draw [fill=ffwwqq] (0,1.41) circle (2.5pt);
\draw [fill=ffwwqq] (-1.22,0.71) circle (2.5pt);
\draw [fill=ffwwqq] (1.22,0.71) circle (2.5pt);
\draw [fill=ffwwqq] (1.22,-0.71) circle (2.5pt);
\draw [fill=ffwwqq] (0,-1.41) circle (2.5pt);
\draw [fill=ffwwqq] (-1.22,-0.71) circle (2.5pt);
\draw [fill=uuuuuu] (-1.5845,-1.5845) circle (2pt);
\draw[color=uuuuuu] (-2.2,-1.5845) node {\large $c_i(a)$};
\draw[color=qqwuqq] (0.3,0.9) node {\large $\pi$};
\end{scriptsize}
\end{tikzpicture}
    \caption{Example on how to define $c_i(a)$ and $r_i$ for $\lambda = 6$. We pick $r_i$ such that lines $c_i(a)b$ and $ab$ are both disjoint from the circle centered around $p_i$ with radius $r_i$. The blue points represent Maker's points, while the orange points are Breaker's points.}
    \label{fig: defining ci}
\end{figure}

\begin{theorem}\label{thm:bichromatic}
In the bichromatic $1:12$ Maker-Breaker game with $k\geq 8$, Breaker can win.
\end{theorem}
\begin{proof}
Assume without loss of generality that $k=8$. We prove, by contradiction, that at the end of Breaker's move, the interior of every 7-hole formed by points in $M$ contains some point in $B$. This implies that Maker cannot create an 8-hole in the next round.

Breaker follows the strategy described above for $s=12$ (that is, $\lambda=6$). Recall that Breaker places points near the vertices of a regular hexagon inscribed in a circle $C_i$ centered at each point $p_i\in M$. The central angle between two consecutive vertices of a regular hexagon is $2\pi/6=\pi/3$. A perturbation moves the vertices along the circle $C_i$ such that one of the angles between two consecutive points exceeds $\pi/3$. However, the perturbation guarantees the following:
\begin{quote}
    If $p_i,p_x,p_y\in M$, such that $\angle p_xp_ip_y>\pi/3$ and Maker placed $p_x$ before $p_i$ and $p_y$, then one of Breaker's points on the circle $C_i$ lies in the angular domain $\angle p_xp_ip_y$. \hfill $(\star)$
\end{quote}
This can be seen in the following way. The perturbation moves all vertices of the hexagon counterclockwise by at most angle $\varepsilon_i$. If $\angle p_xp_ip_y\geq \pi/3+\varepsilon_i$, then the angular domain will contain at least one of the perturbed points. Assume that $\pi/3< \angle p_xp_ip_y<\pi/3+\varepsilon_i$. Since $\varepsilon_i<\pi/3$, a counterclockwise perturbation moves at most one vertex of the hexagon out of $\angle p_xp_ip_y$.
After the perturbation, only $\angle q_6 p_i q_1$ exceeds $\pi/3$, so we may assume that
$\angle p_xp_ip_y$ contains exactly one vertex of the regular hexagon, namely $q_1^0$.
Suppose, for contradiction, that $q_1$ lies outside of $\angle p_xp_ip_y$.
Since $\varepsilon_i$ was chosen such that if Maker placed both $p_x$ and $p_y$ before $p_i$, then $q_1$ and $q_1^0$ are on the same side of both lines $p_ip_x$ and $p_ip_y$, we have that $q_1$ lies in $\angle p_xp_ip_y$. Hence, we may assume that Maker placed these three points in the order $p_x,p_i,p_y$. The choice of $\varepsilon_i$ ensures that $q_1$ and $q_1^0$ are on the same side of the line $p_xp_i$. Consequently, $q_1^0q_1$ crosses $p_ip_y$, and so $\angle q_1^0 p_ip_y<\angle q_1^0 p_i q_1 = \epsilon_i$.
Under these assumptions, $p_6^0$, $p_x$, $q_1^0$, and $p_y$ are in this counterclockwise order around $p_y$; see \Cref{fig:angle explanation}.
In particular, we have
$\angle q_6^0 p_i p_x +\angle p_x p_i p_y
= \angle q_6^0 p_i q_1^0 + \angle q_1^0 p_i p_x$,
which in turn gives
$\angle q_6^0 p_i p_x
=\angle q_6^0 p_i q_1^0 + \angle q_1^0p_ip_x- \angle p_x p_i p_y
< \pi/3 +\varepsilon_i-\pi/3=\varepsilon_i$.
However, Breaker placed $q_6^0$ such that $\angle p_x p_i q_6^0 > \varepsilon_i$. This is a contradiction, proving $(\star)$.

\begin{figure}
    \centering
\begin{tikzpicture}[line cap=round,line join=round,>=triangle 45,x=2cm,y=2cm]
\clip(-2.5,-0.25) rectangle (2.5,3);

\draw [line width=1pt] (-0.64,2.63)-- (0.00,0.00);
\draw [line width=1pt] (2,1.47)-- (0.00,0.00);
\draw [line width=1pt] (0.00,2.00)-- (0.00,0.00);
\draw [line width=1pt] (0.00,0.00)-- (1.81,0.87);

\draw [shift={(0.00,0.00)},line width=1pt,color=qqwuqq,->]
plot[domain=0.47:0.6,variable=\t]
({1.8*cos(\t r)},{1.8*sin(\t r)});

\draw [shift={(0.00,0.00)},line width=1pt,color=qqwuqq,->]
plot[domain=0.7:1.75,variable=\t]
({cos(\t r)},{sin(\t r)});

\draw [shift={(0.00,0.00)},line width=1pt,color=pink,->]
plot[domain=0.51:1.5,variable=\t]
({1.4*cos(\t r)},{1.4*sin(\t r)});

\draw [shift={(0.00,0.00)},line width=1pt,color=pink,->]
plot[domain=1.6:1.8,variable=\t]
({1.8*cos(\t r)},{1.8*sin(\t r)});

\draw [shift={(0.00,0.00)},line width=1pt,dotted]
plot[domain=0.00:2.61,variable=\t]
({2*cos(\t r)},{2*sin(\t r)});

\draw [fill=ududff] (0.00,0.00) circle (2.50pt);
\draw[color=ududff] (-0.25,0) node {\large $p_i$};

\draw [fill=ududff] (-0.64,2.63) circle (2.50pt);
\draw[color=ududff] (-0.9,2.6) node {\large $p_y$};

\draw [fill=ffwwqq] (0.00,2.00) circle (2.50pt);
\draw[color=ffwwqq] (0,2.2) node {\large $q_1^0$};

\draw [fill=ududff] (2,1.47) circle (2.50pt);
\draw[color=ududff] (1.9,1.65) node {\large $p_x$};

\draw [fill=ffwwqq] (1.81,0.87) circle (2.50pt);
\draw[color=ffwwqq] (1.85,1.1) node {$q_6^0$};

\draw[color=black] (0.60,2.06) node {\large $C_i$};

\end{tikzpicture}
\caption{Visualization of the
equation $\angle q_6^0 p_i p_x +\angle p_x p_i p_y = \angle q_6^0 p_i q_1^0 + \angle q_1^0 p_i p_x$.
The left hand side is drawn with green angles, and the right hand side with pink ones.}
\label{fig:angle explanation}
\end{figure}

Consider the first round in which, at the end of Breaker's move, there is a 7-hole $H\subset M$ whose convex hull does not contain any point from $B$.
\begin{claim}
The polygon $\conv(H)$ has four consecutive vertices $q_1, q_2, q_3, b$ such that the interior angles of $\conv(H)$ at $q_1$, $q_2$, and $q_3$ are each at least $\pi/3$.
\end{claim}
\begin{proof}
    Let the interior angles in $H$ be $\theta_1\le  \dots\le  \theta_7$, noting that these are all less than $\pi$ and sum to $\theta_1+\dots +\theta_7=5\pi$ (since in any convex $t$-gon, the interior angles sum to $(t-2)\pi$). This implies that at most two of $\theta_1,\dots, \theta_7$ are $\le \pi/3$ (or else $\theta_1+\theta_2 +\theta_3\le 3\pi/3\le \pi$ and $\theta_4+ \theta_5+\theta_6+\theta_7< 4\pi$ contradicts $\theta_1+\dots +\theta_7=5\pi$). Let $a$ and $b$ be the points with the two smallest angles $\theta_1$ and $\theta_2$. One of the two arcs of $H$ joining $a$ and $b$ must contain at least $\lceil 5/2\rceil= 3$ of the remaining $5$ vertices. Letting $q_1, q_2, q_3$ be three consecutive vertices before $b$ on this arc proves the claim.
\end{proof}

We distinguish between two cases based on the order in which Maker placed $q_1$, $q_2$, and $q_3$.

\begin{itemize}
\item[] \emph{Case 1: $q_2$ is the last point in $\{q_1,q_2,q_3\}$ placed by Maker.}
Let $i$ be the round in which Maker placed $q_2$, that is, $p_i=q_2$. Property $(\star)$ implies that the angular domain $\angle q_1 q_2 q_3$ contains some point $v\in B\cap C_i$.  By the choice of $r_i$, the line $q_1q_3$ is disjoint from the circle $C_i$. This implies that $v$ is in the interior of $\Delta q_1q_2q_3 \subset  \mathrm{conv}(H)$.

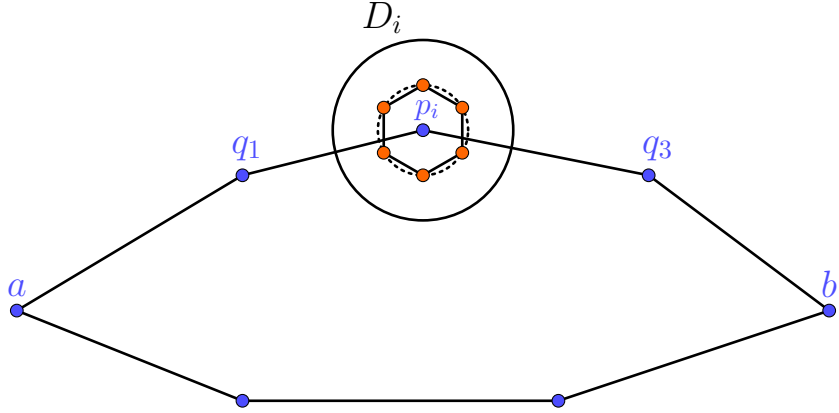
\begin{figure}
    \centering
    \begin{tikzpicture}[line cap=round,line join=round,>=triangle 45,x=1cm,y=1cm, scale=0.6]
\clip(-11,-3) rectangle (11,7);
\draw [line width=1pt] (-4,3)-- (0,4);
\draw [line width=1pt] (0,4)-- (5,3);
\draw [line width=1pt] (5,3)-- (9,0);
\draw [line width=1pt] (9,0)-- (3,-2);
\draw [line width=1pt] (3,-2)-- (-4,-2);
\draw [line width=1pt] (-4,-2)-- (-9,0);
\draw [line width=1pt] (-9,0)-- (-4,3);
\draw [line width=1pt] (0,4) circle (2cm);
\draw [line width=1pt,dotted] (0,4) circle (1cm);
\draw [line width=1pt] (0,5)-- (0.87,4.5);
\draw [line width=1pt] (0.87,4.5)-- (0.87,3.5);
\draw [line width=1pt] (0.87,3.5)-- (0,3);
\draw [line width=1pt] (0,3)-- (-0.87,3.5);
\draw [line width=1pt] (-0.87,3.5)-- (-0.87,4.5);
\draw [line width=1pt] (-0.87,4.5)-- (0,5);
\begin{scriptsize}
\draw [fill=ududff] (9,0) circle (4pt);
\draw[color=ududff] (9,0.6) node {\Large $b$};
\draw [fill=ududff] (-9,0) circle (4pt);
\draw[color=ududff] (-9,0.5) node {\Large $a$};
\draw [fill=ududff] (-4,-2) circle (4pt);
\draw [fill=ududff] (3,-2) circle (4pt);
\draw [fill=ududff] (5,3) circle (4pt);
\draw[color=ududff] (5.2,3.6) node {\Large $q_3$};
\draw [fill=ududff] (0,4) circle (4pt);
\draw[color=ududff] (0.1,4.45) node {\large $p_i$};
\draw [fill=ududff] (-4,3) circle (4pt);
\draw[color=ududff] (-3.9,3.6) node {\Large $q_1$};
\draw[color=black] (-0.9,6.5) node {\Large $D_i$};
\draw [fill=ffwwqq] (0,5) circle (4pt);
\draw [fill=ffwwqq] (0.87,4.5) circle (4pt);
\draw [fill=ffwwqq] (0.87,3.5) circle (4pt);
\draw [fill=ffwwqq] (0,3) circle (4pt);
\draw [fill=ffwwqq] (-0.87,3.5) circle (4pt);
\draw [fill=ffwwqq] (-0.87,4.5) circle (4pt);
\end{scriptsize}
\end{tikzpicture}
    \caption{Picture of Breaker's strategy (orange) if $q_2$ is the last point added by Maker (blue). Breaker places at least one point in the interior of the 7-hole.}
    \label{fig: case one bichrom}
\end{figure}

\item[] \emph{Case 2: $q_1$ or $q_3$ is the last point in $\{q_1,q_2,q_3\}$ placed by Maker.} Assume without loss of generality that $q_3$ is the last point in $\{q_1,q_2,q_3\}$. Let $i$ be the round in which Maker placed $q_3$, that is, $p_i=q_3$.
Property $(\star)$ implies that the angular domain $\angle q_2 q_3 b$ contains some point $v\in B\cap C_i$. We further distinguish between two subcases.
    \begin{itemize}
    \item[] \emph{Case 2a: Maker placed $b$ before $q_3$.} By the choice of $r_i$, the circle $C_i$ is disjoint from the line $q_2b$.
    \item[] \emph{Case 2b: Maker placed $b$ after $q_3$.}
    Our goal is to show that $C_i$ is disjoint from the line $q_1b$ in this case. Suppose, for the sake of contradiction, that the line $q_1b$ intersects $C_i$.
    Breaker's strategy guarantees that $b\notin D_i$ (indeed, if $b\in D_i$, when Maker placed $b$, then $D_i$ was replaced by a smaller disk). Recall that $\pi/3\leq \angle q_2 q_3 b <\pi$. Let $x:=b$. We continuously move $x$ along the line $bp_i$ to the first intersection point between $bp_i$ and the boundary of $D_i$, and then continuously move $x$ along the boundary of $D_i$ while increasing the angle $\angle q_2 q_3 x$ to $\pi$.
    In both portions of the continuous motion, the line $q_2x$ intersects $C_i$ at all times; see \Cref{fig:case 2b}.
    At the end of the continuous motion, we have $x=c_i(q_2)$.
    However, by the choice of the radius $r_i$, the line $q_1 c_i(q_2)$ is disjoint from $C_i$: a contradiction. This shows that the line $q_1b$ is disjoint from $C_i$.
    \end{itemize}

\begin{figure}
    \centering
    \begin{tikzpicture}[line cap=round,line join=round,>=triangle 45,x=1cm,y=1cm]
\clip(0,-1) rectangle (9,5);
\draw [shift={(5,2)},line width=1pt,color=qqwuqq,fill=qqwuqq,fill opacity=0.1] (0,0) -- (166:0.4) arc (166:346:0.4) -- cycle;
\draw [line width=1pt] (1,3)-- (5,2);
\draw [line width=1pt] (5,2)-- (8,0);
\draw [line width=1pt,dotted] (1,3)-- (8,0);
\draw [line width=1pt,dashed] (5,2) circle (1cm);
\draw [line width=1pt] (4.7,3)-- (5.7,2.7);
\draw [line width=1pt] (5.7,2.7)-- (6,1.8);
\draw [line width=1pt] (6,1.8)-- (5.3,1);
\draw [line width=1pt] (5.3,1)-- (4.3,1.3);
\draw [line width=1pt] (4.3,1.3)-- (4,2.2);
\draw [line width=1pt] (4,2.2)-- (4.7,3);
\draw [line width=1pt, dashed] (7.13,1.47)-- (5,2);
\draw [line width=1pt] (5,2) circle (2.2cm);
\begin{scriptsize}
\draw [fill=ududff] (1,3) circle (2.5pt);
\draw[color=ududff] (1.1,3.3) node {\large $q_2$};
\draw [fill=ududff] (5,2) circle (2.5pt);
\draw[color=ududff] (5.1,2.3) node {\large $q_3$};
\draw [fill=ududff] (8,0) circle (2.5pt);
\draw[color=ududff] (8.1,0.3) node {\large $b$};
\draw[color=black] (4,2.8) node {\large $C_i$};
\draw [fill=ffwwqq] (4.7,3) circle (2.5pt);
\draw [fill=ffwwqq] (4,2.2) circle (2.5pt);
\draw [fill=ffwwqq] (4.3,1.3) circle (2.5pt);
\draw [fill=ffwwqq] (5.3,1) circle (2.5pt);
\draw [fill=ffwwqq] (6,1.8) circle (2.5pt);
\draw [fill=ffwwqq] (5.7,2.7) circle (2.5pt);
\draw[color=black] (3.5,4) node {\large $D_i$};
\draw [fill=uuuuuu] (7.13,1.47) circle (1.75pt);
\draw[color=uuuuuu] (7.4,1.5) node {\large $x$};
\draw[color=qqwuqq] (4.4,1.9) node {\large $\pi$};
\end{scriptsize}
\end{tikzpicture}
    \caption{We start with $x = b$ and slide the point in two continuous motions. First we move it along $bp_i$ until it reaches the boundary of $D_i$, and then along this boundary until $\angle q_2 q_3 x = \pi$.}
    \label{fig:case 2b}
\end{figure}
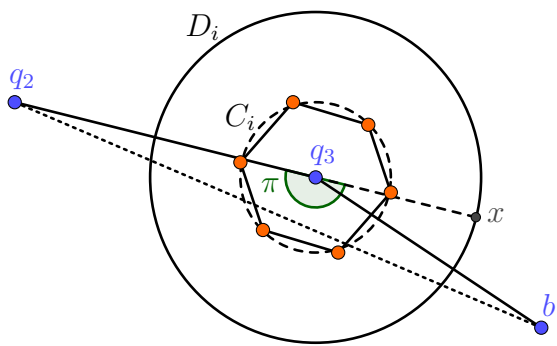

In both subcases (of Case~2), we have shown that the line $q_1b$ is disjoint from the circle $C_i$. Consequently, $v$ is in the interior of the convex quadrilateral $\mathrm{conv}\{q_1,q_2, q_3, b\}$. Since $\mathrm{conv}\{q_1,q_2, q_3, b\}\subset \mathrm{conv}(H)$, then $v$ is in the interior of $\mathrm{conv}(H)$.
\end{itemize}

In both cases, we have found a point $v\in B$ in the interior of $\mathrm{conv}(H)$, contradiction.
\end{proof}

The proof of \Cref{thm:bichromatic} generalizes and yields a trade-off between $k$ and $s$, for all $s\geq 6$.

\begin{theorem}\label{thm:bichromatic+}
For every integer $\lambda\geq 3$, Breaker wins the bichromatic Maker-Breaker game
with $1:2\lambda$ bias if $k\geq 3\lceil 2\lambda/(\lambda-2)\rceil-1$.
In particular, if $k\geq 17$ for $\lambda =3$;
$k\geq 11$ for $\lambda\in \{4,5\}$; and
$k\geq 8$ for $\lambda\geq 6$.
\end{theorem}
\begin{proof}
Let $\lambda\in \{3,4,5,6\}$ and assume without loss of generality that $k=3\lceil 2\lambda/(\lambda-2)\rceil-1$. We prove that at the end of Breaker's move, the interior of every $(k-1)$-hole formed by $M$ contains some point in $B$. This implies that Maker cannot create a $k$-hole in the next round.

Breaker follows the strategy described above for $s=2\lambda$. That is, Breaker places points at the vertices of a regular $\lambda$-gon inscribed in some circles around Maker points. The central angle between two consecutive vertices of a regular $\lambda$-gon is $2\pi/\lambda$.   The supplementary angle of $2\pi/\lambda$ is $\pi-2\pi/\lambda=\frac{\lambda-2}{\lambda}\pi$. By the pigeonhole principle, at most $\lceil 2\lambda/(\lambda-2)\rceil-1$ supplementary angles exceed $\frac{\lambda-2}{\lambda}\pi$, so at most this many interior angles are less than $2\pi/\lambda$ in any convex polygon.

For the sake of contradiction, consider the first round in which, at the end of Breaker's move, there is a $(k-1)$-hole $H\subset M$ whose convex hull does not contain any point from $B$.
The $\lceil 2\lambda/(\lambda-2)\rceil-1$ vertices of $\mathrm{conv}(H)$ with the smallest interior angles decompose the boundary of $\mathrm{conv}(H)$ into closed Jordan arcs. By the pigeonhole principle, one of the arcs contains at least 3 points of $H$ in its interior. Denote by $q_1, q_2,q_3$ three consecutive vertices in the interior of this arc, preceded and succeeded by vertices $a$ and $b$, respectively. Then the interior angles at $q_1,q_2,q_3$ are each at least $2\pi/\lambda$.

The proof is completed by a case distinction identical to the one in the proof of \Cref{thm:bichromatic}, with regular $\lambda$-gons and $2\pi/\lambda$ in place of hexagons and $\pi/3$.
\end{proof}

\subsection{One-Round Bichromatic Game with Bias in Favor of Maker}
\label{ssec:HalesJewett}

In contrast to \Cref{obs:bichrom}, we show that Maker wins the one-round bichromatic game if she has a bias. The analysis of Maker's winning strategy uses the density Hales-Jewett theorem by Furstenberg and Katznelson~\cite{FK91} (see also ~\cite{DKT14,Gow10,Polymath9} for simplified proofs). In the statement of the theorem below, we use the notation $[t]=\{1,2,\ldots, t\}$, where $[t]^d$ is a section of the $d$-dimensional integer lattice $\mathbb{Z}^d$.

\begin{theorem}[Density Hales-Jewett Theorem~\cite{FK91}]\label{thm:HalesJewett}
For every integer $t\geq 2$ and every $\delta\in (0,1]$, there exists an integer $d_0$ with the following property. If $R\subseteq [t]^{d_0}\subset \mathbb{Z}^{d_0}$ and $|R|/t^{d_0}\geq \delta$, then $R$ contains $t$ collinear points that form a combinatorial line, which means that we can form a sequence form the $t$ points where each coordinate is either increasing, or constant.
\end{theorem}

Several remarks are in order.
(1) For our purposes, it would make sense to replace combinatorial lines with geometric lines (where each coordinate is increasing, decreasing or constant); the smallest $t$ with this weaker property is known as the Moser number~\cite{Moser70,Polymath2010}. However, the Density Hales-Jewett theorem gives the current best bounds for Density Moser numbers,
so we use this theorem directly.
(2) The connection between the Density Hales-Jewett theorem and Maker's strategy is a simple and well-know idea: a suitable generic projection of $[t]^d$ onto a plane. This idea was previously used (among others) by Alon~\cite{Alon12} to obtain a super-linear lower bound for planar (weak) $\varepsilon$-nets for lines, and even earlier by Pach, Tardos and T\'{o}th \cite{PachTT09} to show that the chromatic number of planar point-line incidence hypergraphs can be arbitrarily large.
(3) In fact, our problem is strongly related to the weak $\varepsilon$-net problem for convex sets in the plane in the special case that Maker's point set $M$ is in convex position. A key difference between these problems is that a (weak) $\varepsilon$-net for $M$ can include points from $M$, but in the Erd\H{o}s-Szekeres game, Breaker has to choose a set $B$ of points such that $M\cap B=\emptyset$ and $M\cup B$ is in general position.

Here, beyond a generic linear transformation, we will also need to ``bend'' lines of given directions in a suitable way.
In particular, our contribution is an intricate perturbation that transforms collinear $t$-sets into $t$-holes. We give a short summary before describing Maker's strategy: we first create three ``bundles'' of collinear lines, the bundles are each almost parallel to three distinct directions (at angular distance $2\pi/3$ apart), and then perturb each bundle independently to bundles of parabolas. Each $t$-set on a parabola in a bundle forms a long and thin $t$-hole. Importantly, a typical Maker point $p$ participates in such $t$-holes in three different directions, so Breaker needs at least two points in some small $\delta$-neighborhood of $p$,
or else $p$ can contribute to any $t$-hole in one of the three bundles.
Furthermore, any Breaker point that is not in a $\delta$-neighborhood of a Maker point can block at most two such long and thin $t$-holes. Finally, an easy averaging argument shows that, with $1:(2-\varepsilon)$ bias, sufficiently many Maker points ``survive'' (in the sense that they each have at most one Breaker point in their $\delta$-neighborhood) and can form a $\frac{t+1}{2}$-hole.

We make an easy observation that we use for Maker's strategy.
\begin{observation}
\label{obs:grid}
Let $d,k\in \mathbb{N}$, $t=2k-1$ and consider $[t]^d\subset \mathbb{Z}^d$.
If $\ell_1$ and $\ell_2$ are distinct lines in $\mathbb{R}^d$ that each contain $t$ points of $[t]^d$, then $\ell_1\cap \ell_2$ is either empty or a point in $[t]^d$.
\end{observation}
\begin{proof}
Let $\ell_1$ and $\ell_2$ be distinct lines spanned by two collinear $t$-sets in $[t]^d$, and assume that $\ell_1 \cap \ell_2\neq \emptyset$. Then $\ell_1$ and $\ell_2$ are not parallel, and so their (unit) direction vectors differ in some coordinate $i\in [d]$. Consider the orthogonal projection of $\ell_1$ and $\ell_2$ into the coordinate plane spanned by the standard basis vectors $\mathbf{e}_i$ and $\mathbf{e}_j$, for some $j\in [d]\setminus \{i\}$. The projection of $[t]^d$ is the 2-dimensional grid $[t]^2$; and the projection of $\ell_1$ (resp., $\ell_2$) is either an axis-parallel grid line or a diagonal of the square $[1,t]^2$. In any case, the orthogonal projection of $\ell_1\cap \ell_2$ is in $[t]^2$ (e.g., if both projections are diagonals of $[1,t]^2$, the intersection is $(\frac{t+1}{2},\frac{t+1}{2})=(k,k)\in [t]^2$ since $t=2k-1$). Consequently, all $d$ coordinates of $\ell_1\cap \ell_2$ are integers in $[t]$. This completes the proof of the claim.
\end{proof}

\begin{theorem}\label{thm:grid}
    For every $\varepsilon>0$ and integer $k\geq 3$, Maker wins in the bichromatic one-round Maker-Breaker game with bias $1:(2-\varepsilon)$.
\end{theorem}

\begin{proof}
Given $\varepsilon>0$ and an integer $k\geq 3$, we present Maker's strategy.
For the parameters $t=2k-1$ and $\delta=\varepsilon/6$, assume that \Cref{thm:HalesJewett} yields the integer $d_0$; let $d=3{d_0}$.

We partition the coordinates $\{1,\ldots,d\}$ into three subsets of equal size:
$D_0=\{1,\ldots,d_0\}$,
$D_1=\{d_0+1,\ldots,2d_0\}$, and
$D_2=\{2d_0+1,\ldots,d\}$.
For every $j\in \{0,1,2\}$, denote by $\mathbb A_j$ the $d_0$-dimensional subspace of $\mathbb{R}^d$ spanned by the standard basis vectors $\{\mathbf{e}_i:i\in D_j\}$; thus, $\mathbb{R}^d=\mathbb A_0 \oplus \mathbb  A_1 \oplus \mathbb A_2$.
It will be convenient to write a point $z\in \mathbb{R}^d$ as the vector sum of its orthogonal projections to these three subspaces: $z=z_0+z_1+z_2$ where $z_j\in \mathbb A_j$.

Fix a plane $P$ with a coordinate system $(x,y)$ and a generic linear transformation $\pi: \mathbb{R}^{d_0}\to P$ with the property that 
$\pi(\mathbf{e}_i)$ is in a small neighborhood of $(1,0)$ for every standard basis vector $\mathbf{e}_i\in \mathbb{R}^{d_0}$. For simplicity, we may even choose $\pi$ so that $\pi(\mathbf{e}_i)$ is a unit vector whose angle with $(1,0)$ is, say, at most $\pi/20$.
That is, the image of the cone $\Lambda=\{\sum_i \lambda_i\mathbf{e}_i:\lambda_i\ge 0\}$ is quite ``flat'', close to horizontal.
The direction of a combinatorial line of $[t]^{d_0}$ is always from $\Lambda$.
We also define the following perturbation in the plane $P$:
\[
\tau(x,y)=(x,y+\gamma x^2),
\]
for a sufficiently small $\gamma>0$ specified later.
Note that $\tau$ maps every nonvertical line to a parabola.
Specifically, every halfline contained in $\Lambda$ is mapped by $\tau\circ\pi$ into a flat halfparabola that barely curves.
Finally, we define $\varphi: \mathbb R^d \to P$ and $\varphi^\tau: \mathbb R^d \to P$ as follows:
\[
\varphi(z)=\sum_{j=0}^2 \omega^j \circ\pi(z_j)
\hspace{1cm}\text{ and }\hspace{1cm}
\varphi^\tau(z)=\sum_{j=0}^2 \omega^j \circ \tau\circ\pi(z_j),
\]
where $\omega$ denotes the rotation around the origin $(0,0)$ by $2\pi/3$ in the plane $P$ (note that $\omega^0$ is the identity map, $\omega^1=\omega$, and $\omega^2=\omega^{-1}$ is a rotation by $4\pi/3$).
In particular, any of $\varphi$ or $\varphi^\tau$ takes any axis-parallel positive unit vector from $\mathbb A_j$ close to $\omega^j(1,0)$ for $j=0,1,2$.
For example, the image of $\mathbf{e}_{d/2}$ is close to $\omega(1,0)=(-1/2,\sqrt 3/2)$, or the direction of the image of any unit vector from $\Lambda$ whose first $2d_0$ coordinates are close to $0$ is close to $\omega^2(1,0)=(-1/2,-\sqrt 3/2)$; see Figure \ref{fig:phi defined}.
The difference between $\varphi$ and $\varphi^\tau$ is that $\varphi$ maps lines to lines, while $\varphi^\tau$ slightly bends them.
We will not need the fact that $\varphi$ and $\varphi^\tau$ are obtained as the sum of $3$ components until later in the proof.

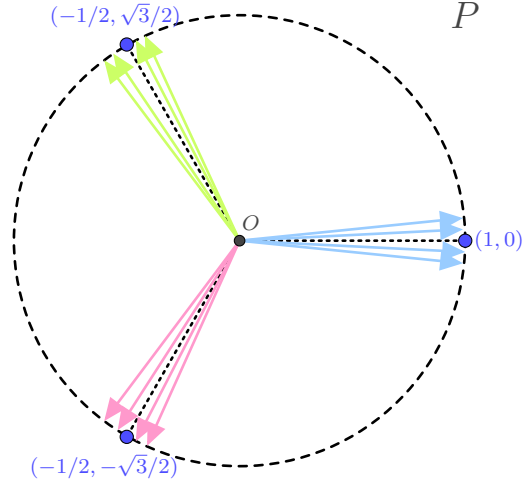
\begin{figure}
    \centering
    \begin{tikzpicture}[line cap=round,line join=round,>=triangle 45,x=3cm,y=3cm]
\clip(-1.61,-1.08) rectangle (1.72,1.09);
\draw [line width=1pt, dashed] (0,0) circle (3cm);
\draw [line width=1pt,dotted] (0,0)-- (-0.50,0.87);
\draw [line width=1pt,dotted] (0,0)-- (-0.50,-0.87);
\draw [line width=1pt,dotted] (0,0)-- (1.00,0);
\draw [->,line width=1pt,color=zzccff] (0,0) -- (0.99,0.05);
\draw [->,line width=1pt,color=zzccff] (0,0) -- (0.99,0.10);
\draw [->,line width=1pt,color=zzccff] (0,0) -- (1.00,-0.05);
\draw [->,line width=1pt,color=zzccff] (0,0) -- (0.99,-0.10);
\draw [->,line width=1pt,color=ccffww] (0,0) -- (-0.46,0.89);
\draw [->,line width=1pt,color=ccffww] (0,0) -- (-0.42,0.91);
\draw [->,line width=1pt,color=ccffww] (0,0) -- (-0.56,0.83);
\draw [->,line width=1pt,color=ffzzcc] (0,0) -- (-0.54,-0.84);
\draw [->,line width=1pt,color=ffzzcc] (0,0) -- (-0.46,-0.89);
\draw [->,line width=1pt,color=ffzzcc] (0,0) -- (-0.60,-0.80);
\draw [->,line width=1pt,color=ccffww] (0,0) -- (-0.60,0.80);
\draw [->,line width=1pt,color=ffzzcc] (0,0) -- (-0.41,-0.91);
\begin{scriptsize}
\draw [fill=ududff] (-0.50,0.87) circle (2.5pt);
\draw[color=ududff] (-0.55,1) node {$(-1/2, \sqrt{3}/2)$};
\draw [fill=ududff] (-0.50,-0.87) circle (2.5pt);
\draw[color=ududff] (-0.6,-1) node {$(-1/2, -\sqrt{3}/2)$};
\draw [fill=ududff] (1.00,0) circle (2.5pt);
\draw[color=ududff] (1.15,0) node {$(1,0)$};
\draw [fill=uuuuuu] (0,0) circle (2pt);
\draw[color=uuuuuu] (0.05,0.08) node {$O$};
\draw[color=uuuuuu] (1,1) node {\Large $P$};
\end{scriptsize}
\end{tikzpicture}
    \caption{$\varphi$ maps the standard basis vectors of $\mathbb{A}_0$, $\mathbb{A}_1$ and $\mathbb{A}_2$ respectively to a small neighborhood of $\omega^0$ (blue), $\omega^1$ (green) and $\omega^2$ (pink). Furthermore, it maps an arbitrary vector $z$ to the sum of the respective rotations of the projections on $P$ of its projections on the three subspaces $\mathbb{A}_j$.}
    \label{fig:phi defined}
\end{figure}

\paragraph{Maker's move.} Now we are ready to describe Maker's move.
Maker places a set $M^\tau$ of $n=(2k-1)^d$ points in the plane given by $M^\tau=\varphi^\tau([t]^d)$, that is, each point of $M^\tau$ is the $\varphi^\tau$-image of a point from the grid $[t]^d\subset \mathbb R^d$.
By the definition of $\varphi^\tau$, the points of $M^\tau$ are very close to the integer lattice generated by the vectors $(1,0)$ and $\omega(1,0)$, with typically several points of $M^\tau$ being very close to the same point of the lattice.
Since $\pi$ is generic, $\varphi^\tau$ is injective on $[t]^d$, and so $\varphi^\tau$ gives a bijection between $[t]^d$ and $M^\tau$.

We also use in our arguments the set $M=\varphi([t]^d)$, i.e., $M$ is $M^\tau$ without the bending caused by $\tau$. As there is a bijection between the points of $M$ and $M^\tau$, we use the following notation: for every point $s=\varphi(z)$, let $s^\tau=\varphi^\tau(z)$.
Similarly, for every set $C\subset M$, let $C^\tau=\{s^\tau:s\in C\}\subset M^\tau$; this operation is consistent with the notation $M^\tau$.
Note that the points of $M$ are very close to the respective points of $M^\tau$, and they form (a section of) a lattice. Since $\pi$ is generic, a $t$-set in $M$ is collinear if and only if its pre-image is a collinear $t$-set in $[t]^d$; and any point in $\mathbb{R}^2\setminus M$ is incident to at most two lines spanned by collinear $t$-sets in $M$.

We specify $\gamma>0$ as follows. Let $L$ be the set of lines spanned by collinear $t$-sets in $M$, and $X$ the set of intersection points of any two distinct lines in $L$. Let $\alpha$ be the minimum distance between a point $x\in X$ and a line $\ell\in L$ with $x\notin \ell$. Then the disks $D(x,\alpha/3)$ of radius $\alpha/3$ centered at $x\in X$ are pairwise disjoint. For any $\beta>0$ and line $\ell\in L$, let $N(\ell,\beta)$ denote the $\beta$-neighborhood of $\ell$, that is $N(\ell,\beta)=\{x\in \mathbb{R}^2: \mathrm{dist}(\ell,x)<\beta\}$. Then $N(\ell,\beta)$ is a strip parallel to $\ell$; and for any pair of nonparallel lines, $\ell_1,\ell_2\in L$, the intersection $N(\ell_1,\beta)\cap N(\ell_2,\beta)$ is a parallelogram centered at $\ell_1\cap \ell_2$. This is illustrated in Figure~\ref{fig:neighborhood parallelogram}.
However, the diameter of this parallelogram depends on the angle between $\ell_1$ and $\ell_2$.  Let $\beta\in (0, \alpha/3)$ be sufficiently small such that for any pair of nonparallel lines $\ell_1,\ell_2\in L$, we have $N(\ell_1,\beta)\cap N(\ell_2,\beta)\subset D(\ell_1\cap \ell_2,\alpha/3)$. Finally, let $\gamma>0$ be sufficiently small such that for the following two conditions are met  for every point $s\in S$ and line $\ell\in L$ with $s\in \ell$:  (a) we have $s^\tau\in N(\ell,\beta)$, and (b) the angle between $\ell$ and the tangent line of the parabola $\tau(\ell)$ at $s^\tau$ is at most $\pi/20$. This completes Maker's strategy.

\begin{figure}
    \centering
\begin{tikzpicture}[line cap=round,line join=round,>=triangle 45,x=1.5cm,y=1.5cm]
\clip(-3,-2) rectangle (3.5,2.5);
\fill[line width=1pt,color=zzffff,fill=zzffff,fill opacity=1] (-9.34,0.5) -- (-9.28,-0.5) -- (11.33,-0.5) -- (11.10,0.5) -- cycle;
\fill[line width=1pt,color=ccffww,fill=ccffww,fill opacity=1] (-4.47,7.10) -- (-3.58,7.31) -- (3.89,-6.06) -- (2.83,-5.95) -- cycle;
\fill[line width=1pt,color=yqqqyq,fill=yqqqyq,fill opacity=1] (-0.78,0.5) -- (0.22,0.5) -- (0.78,-0.5) -- (-0.22,-0.5) -- cycle;
\draw [line width=1pt,domain=-6.36:7.24] plot(\x,0);
\draw [line width=1pt,domain=-6.36:7.24] plot(\x,0.5);
\draw [line width=1pt,domain=-6.36:7.24] plot(\x,-0.5);
\draw [line width=1pt,domain=-6.36:7.24] plot(\x,{(-0.99*\x)/0.55});
\draw [line width=1pt,domain=-6.36:7.24] plot(\x,{(0.5-\x)/0.55});
\draw [line width=1pt,domain=-6.36:7.24] plot(\x,{(-0.5-\x)/0.55});
\draw [fill=uuuuuu] (0,0) circle (3pt);
\draw [fill=xdxdff] (-1.04,1.87) circle (2.5pt);
\draw [fill=ududff] (2.3,0) circle (2.5pt);
\draw [fill=xdxdff] (0.89,-1.6) circle (2.5pt);
\draw [fill=xdxdff] (-2.3,0) circle (2.5pt);
\draw[color=black] (-0.2,2.2) node {$N(\ell_1, \beta)$};
\draw[color=black] (3,0.7) node {$N(\ell_2, \beta)$};
\draw [line width=0.75pt,dotted] (0.05,0.1)-- (0.5,1.05);
\draw[color=uuuuuu] (0.52,1.2) node {$\ell_1 \cap \ell_2$};
\draw [line width=1pt, dashed] (0,0) circle (1.5cm);
\draw [line width=0.75pt,dotted] (-1,0.2)-- (-1.8,0.8);
\draw[color=uuuuuu] (-2,1) node {$D(\ell_1\cap \ell_2,\alpha/3)$};
\end{tikzpicture}
\caption{The intersection of the $\beta$-neighborhoods of the lines spanned by two collinear $t$-sets, $\ell_1$ and $\ell_2$, forms a parallelogram centered at $\ell_1 \cap \ell_2$.
    }
    \label{fig:neighborhood parallelogram}
\end{figure}
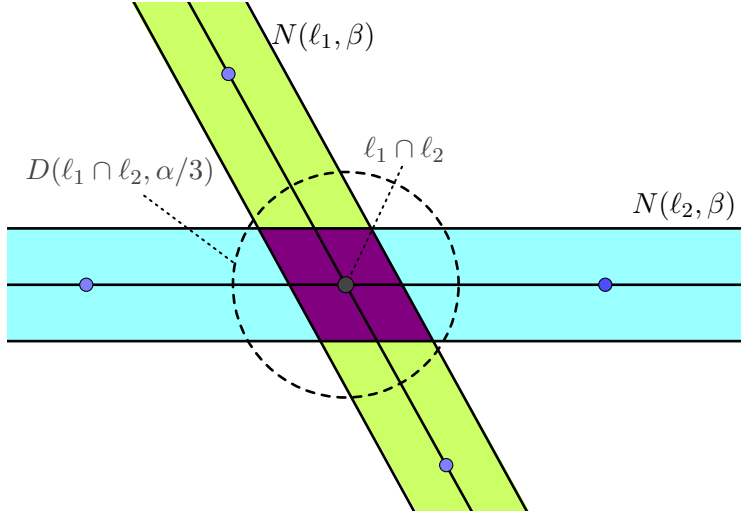

\paragraph{Properties of Maker's points.}
We say that two points $z,z'\in [t]^d$ are \emph{$\mathbb A_j$-siblings} if $z-z'\in \mathbb A_j$. For example, if the first $2d_0$ coordinates of $z$ and $z'$ are the same, then they are $\mathbb A_2$-siblings. Similarly, any $z_0+z_1+z_2$ and $z_0'+z_1+z_2$ are $\mathbb A_0$-siblings if $z_0,z_0'\in \mathbb A_0$. If two points are $\mathbb A_j$-siblings for some $j\in \{0,1,2\}$, then we call them \emph{$\mathbb A$-siblings}.
A collection of points are called \emph{$\mathbb A$-siblings} if they are pairwise $\mathbb A$-siblings. We also extend these terms to their images under $\varphi$ and $\varphi^\tau$: the images of $\mathbb A$-siblings are $\mathbb A$-siblings.
We will argue below that if a $t$-set $C\subset M$ of $\mathbb A$-siblings is collinear, then $C^\tau$ is a $t$-hole with respect to $M^\tau$.

We state a few properties of $M^\tau$. Let $C\subset M$ be a collinear $t$-set of $\mathbb A$-siblings whose pre-images that form a combinatorial line in $[t]^d$.
Then the $t$-set $C^\tau\subset M^\tau$ is contained in a parabola, as $\mathbb A$-siblings are bent in the same direction, hence $C^\tau$ is in convex position. Furthermore, $C^\tau\subset N(\ell,\beta)$, where $\ell$ is the line spanned by $C$, as $s^\tau\in N(\ell,\beta)$ for every vertex $s^\tau\in C^\tau$.
Every point $s\in M$ is contained in at least $d$ lines in $L$. For any two lines $\ell_1,\ell_2\in L$, with $s\in \ell_1\cap \ell_2$,  we have $s^\tau\in N(\ell_1, \beta)\cap N(\ell_2, \beta)$, consequently $s^\tau\in D(\ell_1 \cap \ell_2, \alpha/3)$. If $s\notin \ell$ for some line $\ell\in L$, then $\beta<\alpha/3$ yields $D(s,\alpha/3)\cap N(\ell,\beta)=\emptyset$, which implies $s^\tau\notin N(\ell,\beta)$.
Consequently, the interior of $\conv(C^\tau)$ cannot contain any point of $M^\tau$, and so $C^\tau$ is a $t$-hole. For every $j\in \{0,1,2\}$, let $\mathcal{H}_j$ denote the set of $t$-holes with respect to $M^\tau$ that correspond to collinear $t$-sets of $\mathbb A_j$-siblings in $M$, and let $\mathcal{H}=\bigcup_{j=0}^2\mathcal{H}_j$.

We claim that for any three $t$-holes $C^\tau_0\in\mathcal{H}_0$, $C^\tau_1\in\mathcal{H}_1$, and $C^\tau_2\in\mathcal{H}_2$, the intersection $\conv(C^\tau_0)\cap\conv(C^\tau_1)\cap \conv(C^\tau_2)$ is either empty or consists of a single point $s^\tau\in C^\tau_0\cap C^\tau_1\cap C^\tau_2\subset M^\tau$.
We will refer to this as the \textit{no-colorful-triple-intersection property}; see Figure \ref{fig:no colorful triple}.

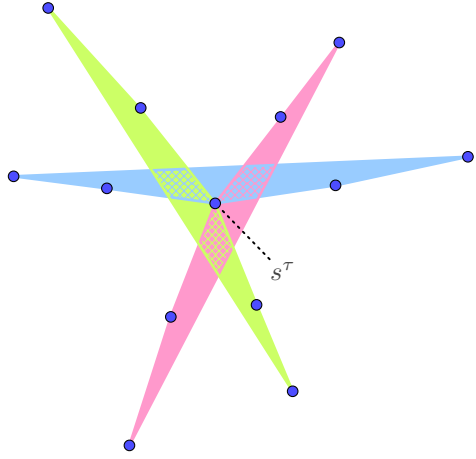
\begin{figure}
    \centering
    \begin{tikzpicture}[line cap=round,line join=round,>=triangle 45,x=1cm,y=1cm]
\clip(-3,-3.5) rectangle (3.5,3);

\fill[line width=0.5pt,color=ffzzcc,fill=ffzzcc,fill opacity=1]
  (-1.14,-3.22) -- (-0.59,-1.51) -- (0,0) -- (0.89,1.15) -- (1.65,2.14) -- cycle;

\fill[line width=0.5pt,color=zzccff,fill=zzccff,fill opacity=1]
  (0,0) -- (-1.44,0.2) -- (-2.68,0.36) -- (3.36,0.62) -- (1.6,0.24) -- cycle;

\fill[line width=0.5pt,color=ccffww,fill=ccffww,fill opacity=1]
  (-2.22,2.6) -- (-1.01,1.27) -- (0,0) -- (0.55,-1.35) -- (1.03,-2.5) -- cycle;

\fill[line width=0.5pt,color=ffzzcc,fill=ffzzcc,fill opacity=1]
  (0,0) -- (0.23,-0.58) -- (0.04,-0.95) -- (-0.22,-0.55) -- cycle;

\fill[line width=1pt,color=ccffww,fill=ccffww,pattern=north east lines,pattern color=ccffww]
  (0,0) -- (0.23,-0.58) -- (0.04,-0.95) -- (-0.22,-0.55) -- cycle;
\fill[line width=1pt,color=ccffww,fill=ccffww,pattern=north west lines,pattern color=ccffww]
  (0,0) -- (0.23,-0.58) -- (0.04,-0.95) -- (-0.22,-0.55) -- cycle;

\fill[line width=1pt,color=ffzzcc,fill=ffzzcc,pattern=north east lines,pattern color=ffzzcc]
  (0,0) -- (0.58,0.09) -- (0.8,0.51) -- (0.38,0.49) -- cycle;
\fill[line width=1pt,color=ffzzcc,fill=ffzzcc,pattern=north west lines,pattern color=ffzzcc]
  (0,0) -- (0.58,0.09) -- (0.8,0.51) -- (0.38,0.49) -- cycle;

\fill[line width=1pt,color=zzccff,fill=zzccff,pattern=north east lines,pattern color=zzccff]
  (-0.62,0.09) -- (0,0) -- (-0.37,0.46) -- (-0.84,0.44) -- cycle;
\fill[line width=1pt,color=zzccff,fill=zzccff,pattern=north west lines,pattern color=zzccff]
  (-0.62,0.09) -- (0,0) -- (-0.37,0.46) -- (-0.84,0.44) -- cycle;

\draw [line width=1pt,color=zzccff] (0,0)-- (-1.44,0.2) -- (-2.68,0.36) -- (3.36,0.62) -- (1.6,0.24) -- cycle;
\draw [line width=1pt,color=ffzzcc] (-1.14,-3.22) -- (-0.59,-1.51) -- (0,0) -- (0.89,1.15) -- (1.65,2.14) -- cycle;
\draw [line width=1pt,color=ccffww] (-2.22,2.6) -- (-1.01,1.27) -- (0,0) -- (0.55,-1.35) -- (1.03,-2.5) -- cycle;

\draw [fill=ududff] (0,0) circle (2pt);
\draw [fill=ududff] (-1.44,0.2) circle (2pt);
\draw [fill=ududff] (-2.68,0.36) circle (2pt);
\draw [fill=ududff] (1.6,0.24) circle (2pt);
\draw [fill=ududff] (3.36,0.62) circle (2pt);
\draw [fill=ududff] (-0.59,-1.51) circle (2pt);
\draw [fill=ududff] (-1.14,-3.22) circle (2pt);
\draw [fill=ududff] (0.87,1.15) circle (2pt);
\draw [fill=ududff] (1.65,2.14) circle (2pt);
\draw [fill=ududff] (-0.99,1.27) circle (2pt);
\draw [fill=ududff] (-2.22,2.6) circle (2pt);
\draw [fill=ududff] (0.55,-1.35) circle (2pt);
\draw [fill=ududff] (1.03,-2.50) circle (2pt);
\draw [line width=0.75pt,dotted] (0.1,-0.1)-- (0.75,-0.77);
\draw[color=uuuuuu] (0.9,-0.9) node {\small $s^\tau$};
\end{tikzpicture}

    \caption{The no-colorful-triple-intersection property ensures that the intersection of the convex hulls of any three $t$-holes $C^\tau_0\in\mathcal{H}_0$, $C^\tau_1\in\mathcal{H}_1$, and $C^\tau_2\in\mathcal{H}_2$ is either empty or one single point $s^\tau$.
    }
    \label{fig:no colorful triple}
\end{figure}

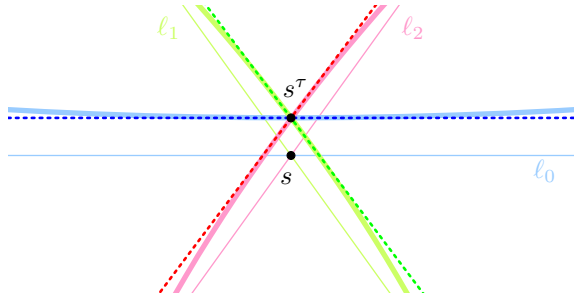
\begin{figure}
    \centering

\begin{tikzpicture}[line cap=round,line join=round,>=triangle 45,x=1.25cm,y=1cm]
\clip(-3,-1.82) rectangle (3,2);

\draw [line width=.5pt,color=ccffww,domain=-3.01:3.29] plot(\x,{(-0--0.87*\x)/-0.5});

\draw [line width=.5pt,color=zzccff,domain=-3.01:3.29] plot(\x,{(-0-0*\x)/0.81});

\draw [line width=.5pt,color=ffzzcc,domain=-3.01:3.29] plot(\x,{(-0-0.87*\x)/-0.5});

\draw [shift={(0,42.97)},line width=2pt,color=zzccff]
    plot[domain=4.56:4.87,variable=\t]
    ({1*42.47*cos(\t r)+0*42.47*sin(\t r)},
    {0*42.47*cos(\t r)+1*42.47*sin(\t r)});

\draw [shift={(20.63,-11.89)},line width=2pt,color=ffzzcc]
  plot[domain=2.47:2.76,variable=\t]
  ({1*24.06*cos(\t r)+0*24.06*sin(\t r)},
   {0*24.06*cos(\t r)+1*24.06*sin(\t r)});

\draw [shift={(-20.31,-11.73)},line width=2pt,color=ccffww]
  plot[domain=0.38:0.67,variable=\t]
  ({1*23.71*cos(\t r)+0*23.71*sin(\t r)},
   {0*23.71*cos(\t r)+1*23.71*sin(\t r)});

\draw [line width=1pt,dotted,color=qqqqff,domain=-3.01:3.29]
  plot(\x,{(21.24-0*\x)/42.47});

\draw [line width=1pt,dotted,color=qqffqq,domain=-3.01:3.29]
  plot(\x,{(6.11-20.31*\x)/12.23});

\draw [line width=1pt,dotted,color=ffqqqq,domain=-3.01:3.29]
  plot(\x,{(6.19--20.63*\x)/12.39});

\draw [fill=black] (0,0.5) circle (1.5pt);
\draw [fill=black] (0,0) circle (1.5pt);
\draw[color = black] (0.05,0.9) node {\small $s^\tau$};
\draw[color = black] (-0.05,-0.3) node {\small $s$};
\draw[color = zzccff] (2.7, -0.2) node {\small $\ell_0$};
\draw[color = ffzzcc] (1.3,1.7) node {\small $\ell_2$};
\draw[color = ccffww] (-1.3,1.7) node {\small $\ell_1$};
\end{tikzpicture}

    \caption{The point $s^\tau$ is the intersection point of three parabolas (drawn by thick): $\tau(\ell_0)$, $\tau(\ell_1)$ and $\tau(\ell_2)$. For each $j \in \{0,1,2\}$, the tangent line of $\tau(\ell_j)$ (drawn by a dotted line), is close to being parallel to $\ell_j$ (drawn by a thin line).
    }
    \label{fig:tangent lines}
\end{figure}

To prove the claim, first notice that $C^\tau_j\subset N(\ell_j,\beta)$
for some line $\ell_j\in L$ that makes an angle at most $\pi/20$ with $\omega^j(0,1)$
for all $j\in \{0,1,2\}$. Observation \ref{obs:grid} implies that $\ell_0\cap \ell_1\cap \ell_2$ is either empty or consists of a single point $s\in M$. The choice of $\beta>0$ implies that $\ell_0\cap \ell_1\cap \ell_2=\emptyset$ yields $\conv(C^\tau_0)\cap\conv(C^\tau_1)\cap \conv(C^\tau_2) =\emptyset$. Assume that the lines intersect at a point $s\in M$.
Then $s^\tau$ is a vertex of $C^\tau_j$ for all $j=0,1,2$, and
\[
    \conv(C^\tau_0)\cap\conv(C^\tau_1)\cap \conv(C^\tau_2) \subset
N(\ell_0,\beta)\cap N(\ell_1,\beta)\cap N(\ell_2,\beta)\subset
D(s,\alpha/3),
\]
In this case, $C^\tau_j$ is contained in a halfplane, bounded by a tangent line of the parabola $\tau(\ell_j)$ at $s^\tau\in \tau(\ell_j)$.
Moreover, the direction of the tangent line is almost $\omega^j(1,0)$; it makes an angle at most $\pi/10$ with $\omega^j(1,0)$. Indeed, by the choice of $\gamma$, it makes an angle at most $\pi/20$ with $\ell_j$, which in turn makes an angle at most $\pi/20$ with $\omega^j(1,0)$. It follows that the halfplane bounded by tangent lines to the parabola $\tau(\ell_j)$ at $s^\tau$ is contained in a (nonconvex) cone $E_j$ of apex angle $\pi/10+\pi+\pi/10=\frac{6\pi}{5}$; where $E_j=\omega^j(E_0)$ for $j=0,1,2$.
The intersection of the three cones, $E_0\cap E_1\cap E_2$, consists of the point $s^\tau$. Since $\conv(C^\tau_j)\subset E_j$ for $j=0,1,2$, then $\conv(C^\tau_0)\cap\conv(C^\tau_1)\cap \conv(C^\tau_2)$ also consists of the point $s^\tau$, where $s^\tau\in C^\tau_0\cap C^\tau_1\cap C^\tau_2\subset M^\tau$, as claimed. See Figure \ref{fig:tangent lines} for a sketch.

\paragraph{Breaker's move.}
Next, Breaker places a set $B$ of at most $(2-\varepsilon)n$ points in the plane, where $n=|M^\tau|$, and $M^\tau\cup B$ is in general position. We show that $M^\tau$ contains a $k$-hole with respect to $M^\tau\cup B$.
We distinguish between two types of points in $B$: let $B^1$ be the set of points in $B$ that lie in the disk $D(s,\alpha/3)$ for some $s\in M$; and let $B^2=B\setminus B^1$.

Every point $b\in B^2$ lies in some neighborhoods $N(\ell,\beta)$ for at most two lines $\ell\in L$ because of the choice of $\alpha$, Observation \ref{obs:grid}, and the fact that $\pi$ is generic. Therefore, every point $b\in B^2$ lies in the convex hull of at most two $t$-holes in $\mathcal{H}$. Let $\mathcal{H}^2$ be the set of $t$-holes $H\in \mathcal{H}$ such that $|\mathrm{conv}(H)\cap B^2|\geq 2$. Double counting yields $|\mathcal{H}^2|\leq |B^2|$.
For each $k$-hole $H\in \mathcal{H}^2$, create a point $b'\in \mathrm{conv}(H)$ that lies in the disk $D(s,\alpha/3)$ for some $s\in H$. Let $B^3$ be the set of all such points $b'$. Since  $|\mathcal{H}^2|\leq |B^2|$, then $|B^3|\leq |B^2|$ and $|B^1\cup B^3|\leq |B|\leq (2-\varepsilon)n$.

Finally, let $M^\tau(B)$ be the set of points $s^\tau \in M^\tau$, such that the disk $D(s,\alpha/3)$ contains at least two points of $B^1\cup B^3$. By construction, $|M^\tau(B)|\leq |B|/2\leq (1-\varepsilon/2)n$.
For every $j\in \{0,1,2\}$, let $M_j^\tau$ be the set of points $s^\tau\in M^\tau\setminus M^\tau(B)$ such that $(B^1\cup B^3)\cap D(s,\alpha/3)$ is disjoint from $\conv(C_j^\tau)$ for every $C_j^\tau\in \mathcal{H}_j$.
We claim that $M^\tau\setminus M^\tau(B) =\bigcup_{j=0}^2 M_j^\tau$.
Indeed, consider a point $s^\tau\in M^\tau\setminus M^\tau(B)$.
If $(B^1\cup B^3)\cap D(s,\alpha/3)=\emptyset$, then $s^\tau \in M_j^\tau$ for all $j\in \{0,1,2\}$. Otherwise, there is a unqiue point $b\in (B^1 \cup B^3)\cap  D(s, \alpha/3)$,
and by the no-colorful-triple-intersection property, there exists a $j\in \{0,1,2\}$ such that $b\notin \conv(C_j^\tau)$ for any $C_j^\tau \in \mathcal{H}_j$.
By averaging, there is a $j\in \{0,1,2\}$ such that $|M^\tau_j|\ge \varepsilon n/6$. Fix such a $j$.

Let $R=\{x\in [t]^d: \varphi^\tau(x)\in M^\tau_j\}$.
Partition $R$ into groups of $\mathbb A_j$-syblings.
Each such group is isomorphic to $[t]^{d/3}=[t]^{d_0}$, and by averaging, one group $R'$ has density at least $\varepsilon/6=\delta$.
The Density Hales-Jewett Theorem implies that the set $R'$ contains a collinear $t$-set; let $T$ be such a set.
Then $H=\varphi^\tau(T)$ is a $t$-hole of $\mathbb A_j$-siblings with respect to $M^\tau$.
While $H$ is not necessarily a $t$-hole with respect to $M^\tau\cup B$, we will prove that it can contain at most one point of $B$.

Suppose that there is some point $b\in B$ in $\mathrm{conv}(H)$.
Since $H\subset M^\tau_j$, then $b$ cannot be in a disk $D(s,\alpha/3)$ for any $s\in M$.
This implies $b\notin B^1$, hence $b\in B^2$. Since $\mathrm{conv}(H)$ contains no points of $B^3$ either, then $\mathrm{conv}(H)$ contains at most one point of $B$.
As a convex set with $t$ vertices and at most one point of $B$ inside it always contains a $(t+1)/2$-hole and $(t+1)/2=k$, we are done.
However, we can then find a subset $H'\subset H$ of size $k<(t+1)/2$ such that $\mathrm{conv}(H)$ does not contain any point of $B$. Specifically, the $t$ points in $H$ have a linear order along a parabola, where $t=2k-1$ yields $(t+1)/2=k$. The first $k$ and the last $k$ points of $H$ in this order determine interior-disjoint convex hulls, so we can take $H'$ as the first or last $k$ points in $H$ such that $b\notin \conv(H')$. Then $H'$ is a $k$-hole with respect to $M^\tau\cup B$, as required.
\end{proof}

\section{Conclusions and Open Problems}
\label{sec:con}

We have considered several Maker-Breaker games that arose from the Erd\H{o}s-Szekeres problem in combinatorial geometry. Firstly, we considered the monochromatic game. In this game, Breaker's points can contribute to the $k$-hole that Maker is trying to obtain. We have shown that if both players play with unit speed (i.e., they each play one point per round), Maker always has a winning strategy. Furthermore, even if Breaker is allowed to play with increased (but constant) speed, Maker still wins.
Due to a result of Conlon and Lim~\cite{ConlonL23}, we know that there exists a sufficiently fast growing function $f:\mathbb{N}\to\mathbb{N}$ such that if Breaker plays $f(i)$ points in round $i$, she wins; but determining the minimum growth rate of such a function $f$
is an intriguing open problem.

The outcome of the game is less somber for Breaker whenever we consider the bichromatic version of the Maker-Breaker game. In this version, Breaker's points can only destroy $k$-holes and cannot contribute to them. When both players have low speed, Maker still wins. This is the case whenever the game is played with bias $1:s$ for any constant $s < 2$ and $k \geq 3$. But if Breaker's speed is high enough, then she is able to win. We showed that for $k \geq 8$ and bias $1:12$ for Maker and Breaker, the latter has a winning strategy. This raises an interesting problem: What is the minimum speed $s(k)$ required for Breaker to win the $1:s(k)$ game? We know that $2\leq s(k)\leq 12$ for $k\geq 8$, and $2\leq s(k)\leq 6$ for $k\geq 17$. However, we do not even know whether Breaker can win with any constant speed $s(k)>1$ when $3<k<8$.
Clearly, $s(k)$ is nonincreasing in $k$, but it remains an open problem to determine $\lim_{k\to \infty}s(k)$, that is, the limit of speeds that allow Breaker to prevent $k$-holes for sufficiently large values of $k$.
Perhaps a first step to address these problems is the observation that Breaker wins the one-round game whenever she has speed two. Whether that also holds for multiple round games is still open, though. The geometric arguments that support Maker's winning strategy using $k$-strips with bias $1:s(k)$, for $s(k) <2$, no longer work for $s(k) \geq 2$. Showing that Maker can still win when Breaker is playing with speed higher than two seems to require novel ideas and a new strategy. In contrast, in the one-round game, when the respective speeds for Maker and Breaker are 1 and $2-\varepsilon$, Maker has a winning strategy based on the Density Hales-Jewett Theorem.

Lastly, another interesting research direction is to determine how efficient the above winning strategies are for Maker. These proposed strategies give an exponential upper bound on the number of rounds required to form a $k$-hole. It is still unclear whether there exist strategies that allow Maker to win in fewer rounds. Proving the optimality of a strategy, and thus proving the exact number of points required, seems to be an even more challenging open problem.

\section{Acknowledgments}
Collaboration on this project started at the \emph{Novi Sad Workshop on Foundations of Computer Science} (\emph{NSFOCS}), held June 22-26, 2024, in Novi Sad, Serbia. The authors thank Mirjana Mikala\v{c}ki, Milo\v{s} Stojakovi\'{c}, Marko Savi\'{c} and Jelena Stratijev for their hospitality and for organizing this successful workshop.


\end{document}